\documentclass[12pt]{amsart}
\usepackage{etex} 
\usepackage{amssymb,amsthm,amsmath,amsfonts,stmaryrd,soul}

\usepackage[OT2,OT1]{fontenc}
\newcommand\cyr
{
\renewcommand\rmdefault{wncyr}
\renewcommand\sfdefault{wncyss}
\renewcommand\encodingdefault{OT2}
\normalfont
\selectfont
}
\DeclareTextFontCommand{\textcyr}{\cyr}

\usepackage{svn}
\SVN $Revision: 528 $

\usepackage{mathabx}
\usepackage{bbm}
\usepackage{blkarray}
\usepackage{amsmath}
\usepackage{graphicx}
\graphicspath{ {./images/} }
\newlength{\tmpl}

\newcommand{\thick}{1.5pt}

\usepackage[colorlinks,anchorcolor=black,citecolor=black,linkcolor=black,filecolor=black,urlcolor=black]{hyperref}

\usepackage{color}

\usepackage{pict2e}

\usepackage[margin=2cm,a4paper]{geometry}

\newcommand{\dd}[1]{d#1}

\usepackage{tikz,pgfplots,pgfplotstable}

\usepgfplotslibrary{external}
\tikzsetexternalprefix{pictures/}
\DeclareSymbolFont{yhlargesymbols}{OMX}{yhex}{m}{n}

\DeclareMathAccent{\widetriangle}{\mathord}{yhlargesymbols}{"E6}

\usepackage[extra]{tipa}
\usepackage{graphicx}
\usepackage{caption,subcaption}

\usepackage{pst-node}
\usepackage{enumerate}

\numberwithin{equation}{section}

\newcommand{\onetwo}[1]{\hat{1}_2^{#1}} 

\DeclareMathOperator{\diag}{\mathrm{diag}}
\DeclareMathOperator{\Tr}{\mathrm{Tr}}

\let\Im\undefined

\DeclareMathOperator{\Im}{\mathrm{Im}}

\newcommand{\Krewl}[1]{{#1}^c}  
\newcommand{\harpleftsign}{\scriptstyle\leftharpoonup}
\newcommand{\harpleft}[2]{%
  \ifx\displaystyle#1\doalign{$\harpleftsign$}{#1#2}\fi
  \ifx\textstyle#1\doalign{$\harpleftsign$}{#1#2}\fi
  \ifx\scriptstyle#1\doalign{\scalebox{.6}[.9]{$\harpleftsign$}}{#1#2}\fi
  \ifx\scriptscriptstyle#1\doalign{\scalebox{.5}[.8]{$\harpleftsign$}}{#1#2}\fi
}

\newcommand{\harprightsign}{\scriptstyle\rightharpoonup}
\newcommand{\harpright}[2]{%
  \ifx\displaystyle#1\doalign{$\harprightsign$}{#1#2}\fi
  \ifx\textstyle#1\doalign{$\harprightsign$}{#1#2}\fi
  \ifx\scriptstyle#1\doalign{\scalebox{.6}[.9]{$\harprightsign$}}{#1#2}\fi
  \ifx\scriptscriptstyle#1\doalign{\scalebox{.5}[.8]{$\harprightsign$}}{#1#2}\fi
}

\newcommand{\doalign}[2]{%
 {\vbox{\offinterlineskip\ialign{\hfil##\hfil\cr#1\cr$#2$\cr}}}%
}

\newcommand{\ED}[1][{}]{E^{\mathcal D}_{#1}}
\newcommand{\EB}[1][{}]{E^{\odot}_{#1}}

\def\R{{\mathbb R}}

\def\C{{\mathbb C}}
\def\IR{{\mathbb R}}
\def\IZ{{\mathbb Z}}
\def\IN{{\mathbb N}}
\def\N{{\mathbb N}}

\def\<{\langle}
\def\>{\rangle}
\def\P{{ P_n}}
\def\state{\varphi} 
\def\Q{{\mathnormal Q}}

\def\J{J_n}

\def\X{{\mathnormal X}}

\def\Y{{\mathnormal Y}}

\def\A{\mathcal{A}}

\def\BL{{{\leftthreetimes}}_\pi}
\def\BR{{{\rightthreetimes}}_\pi}

\newcommand{\abs}[1]{
  \left\lvert
    #1
  \right\rvert}
\newcommand{\bigabs}[1]{
  \bigl\lvert
    #1
  \bigr\rvert}

\DeclareMathOperator{\SG}{\mathfrak{S}} 
\DeclareMathOperator{\NC}{\mathcal{I}} %
\DeclareMathOperator{\ICop}{\mathcal{IC}}  
\DeclareMathOperator{\Sing}{\normalfont{Sing}} %

\newtheorem{th-def}{Theorem-Definition}[section]
\newtheorem{theo}{Theorem}[section]
\newtheorem{lemm}[theo]{Lemma}

\newtheorem{prop}[theo]{Proposition}
\newtheorem{cor}[theo]{Corollary}
\theoremstyle{definition}

\newtheorem{exam}[theo]{Example}
\newtheorem{Rem}[theo]{Remark}

\def\cvput#1[#2]{\pnode(#1,1){#1} \pscircle*(#1,1){.1} \rput(#1,.5){$#2$}}

\usepackage[T1]{fontenc}
\usepackage[latin2]{inputenc}
\usepackage{babel}
\usepackage{floatflt}
\title{The Boolean quadratic forms and tangent law}
\author[Wiktor Ejsmont]{Wiktor Ejsmont}
\author[Patrycja H\k{e}{\'c}ka]{Patrycja H\k{e}{\'c}ka}

\address [Wiktor Ejsmont]{ 
Department of Telecommunications and Teleinformatics, Wroclaw University
of Science and Technology\\
Wybrze\.ze Wyspia\'nskiego 27, 50-370 Wroc\l aw, Poland}
\email{wiktor.ejsmont@gmail.com}
\address [Patrycja H\k{e}{\'c}ka]{Department of Telecommunications and Teleinformatics, Wroclaw University
of Science and Technology\\
Wybrze\.ze Wyspia\'nskiego 27, 50-370 Wroc\l aw, Poland}
\email{patrycja.hecka@gmail.com}


\thanks{ Wiktor
Ejsmont This research was funded in part by Narodowe Centrum Nauki, Poland WEAVE-UNISONO grant 2022/04/Y/ST1/00008.}


\begin{document}
\begin{abstract}   
In  \cite{EjsmontLehner:2020:tangent} we study the limit sums of free commutators and anticommutators and show that the generalized tangent function 
$$
\frac{\tan z}{1-x\tan z}
$$
describes the limit distribution. 
This is the generating function
of the higher order tangent  numbers of
Carlitz and Scoville 
\cite[(1.6)]{CarlitzScoville:1972}
which arose in connection with the enumeration of certain permutations. 
In the present paper we continue to study the limit of weighted  sums of Boolean
commutators and  anticommutators and we  
 show that the shifted generalized tangent  function
appears in  a limit theorem. In order to do this, we shall provide an arbitrary  cumulants formula of the quadratic form. We also apply this result to obtain several results in   a Boolean probability theory. 

\end{abstract}

\date{
  Rev.\SVNRevision,
  \today}

\setlength{\parindent}{0pt}

\maketitle


\section{Introduction}

Free probability was introduced by Voiculescu 30 years ago
\cite{Voiculescu:1986,VoiculescuDykemaNica:1992} in order to solve
some problems in von Neumann algebras of free groups. It has developed into a
whole new field with numerous connections to 
different branches of mathematics such as classical probability,
combinatorics and analysis, in particular random matrices
\cite{Voiculescu:1991}, noncrossing partitions
\cite{NicaSpeicher:2006} and operator algebras. Ejsmont and Lehner ~\cite{EjsmontLehner:2020:tangent}  presented  general free limit theorems   and that article was the starting point
for the present paper.

Speicher and  Woroudi \cite{SpeicherWorudi1995}  introduced another concept of probability, namely Boolean probability.   
In recent years, a number of papers \cite{SpeicherWorudi1995,Lehner2002,Anshelevich2009,Franz2009,Liu2015,Liu2015b, PopaHao2019} have investigated 
theorems for the Boolean convolution of probability measures.  
The key concept of this definition is the notion of noncommutative
Boolean  independence.
As in classical probability where the concept of independence gives rise to classical convolution, the concept of Boolean independence leads to another operation on the
measures on the real line. 
In this paper we show that the limit of mixed sums of commutators and  anti-commutators \begin{align}
\label{eq:sumcommanticomm}
  \frac{\sum_{k<l} a(X_kX_l+X_lX_k) + ib(X_kX_l-X_lX_k)}{n} 
\end{align} in  Boolean probability can be described by the function $$\frac{1}{z}\frac{\tan(bz)}{b-a\tan(bz)}-1.$$
This will follow from a general  limit theorem for arbitrary quadratic forms.
In the way of proof we derive an explicit formula for the Boolean  cumulants of quadratic forms Proposition \ref{prop:CykliczneVariancja},
 which exhibits an interesting
connection to the isomorphism between  interval partitions of $r+1$ elements, with special kind of interval partitions of $2r$ points.

We apply this result to obtain several theorems about quadratic forms  in  Boolean random variables.  In particular, we establish a new proof of the cancellation phenomenon for sample variance and a Boolean version of the $\chi^2$-conjecture. 
Also, among other applications we establish a Boolean analogue of classical theorem of Kagan and Shalaevski \cite{KaganShalaevski}.

It is worth to mention that in Theorem  \ref{lem:equvalencecancellation}, we characterize all quadratic forms with cancellation phenomenon property for i.i.d. random variables. In the solution of this problem, the  zero row sum matrices appear. Such matrices appear 
in mathematics from time to time, for example when studying the Markov chains as $T-I$   where $T$ is the transition matrix for
a finite-state Markov chain. Just as importantly, they appear as the generators for continuous-time Markov
chains  \cite[Ch. 8]{Helms}.
In graph theory, as the Laplacian matrices, they arise in the tree theorems based on the approaches of Kirchoff and Tutte for counting spanning trees, \cite{Bapat,Evans}.

\section{Preliminaries}
\label{sec:prelim}

\subsection{Basic Notation and Terminology}
A  noncommutative probability space is a pair $(\mathcal{A},\state)$
where $\mathcal{A}$ is  a unital $\ast$-algebra, and  $\state:\mathcal{A} \to
\mathrm{\C}$ is a normal state on $\mathcal{A} $.

The elements $\X\in{\mathcal{A}}$
are called (noncommutative) random variables. In the present paper
all random variables are assumed to be self-adjoint.
Given a noncommutative random variable $\X\in{\mathcal{A}}_{sa}$,
the spectral theorem provides a unique probability measure $\mu_X$ on 
$\R$ which encodes the distribution of $\X$ in the state $\state$,
i.e., 
 $\state( f(\X))=\int_\R f(\lambda)\,\dd{\mu_\X(\lambda)}$
for any bounded Borel function $f$ on $\R$.

\subsection{Boolean Independence}
A family of algebras $\left(\mathcal{A}_i\right)_{i\in I}$ 
of $\mathcal{A}$ 
is called \emph{Boolean} independent
 if $\state(\X_{1} \dots \X_{n} ) = \state(\X_{1}) \dots \state(\X_{n} )$, whenever  $\X_{j} \in \mathcal{A}_{i(j)}$ with $i(1)\neq i(2)\neq \dots \neq i(n)$.
Random variables $\X_{1},\dots ,\X_{n} $  are Boolean independent  if the
subalgebras they generate are Boolean independent. 
This kind of product can be traced back to the work of Bo\.zejko \cite{Bozejko1986}. 
For more details about the Boolean probability theory 
see the standard references
\cite{BozejkoLeinertSpeicher,SpeicherWorudi1995,Lehner2002,Anshelevich2009,Franz2009,Liu2015,Liu2015b, PopaHao2019}.

\subsection{Boolean Convolution and the Cauchy-Stieltjes Transform}
It can be shown that the joint distribution of Boolean random variables $X_i$
is uniquely determined by the distributions of the individual random variables
$X_i$ and therefore the operation of \emph{Boolean convolution} is well defined.
Let $\mu$ and $\nu$ be probability measures on $\R$, and
$\X,\Y$ self-adjoint Boolean random variables with respective distributions
$\mu$ and $\nu$.
The distribution of $\X+\Y$ is called the Boolean additive convolution of $\mu$
and $\nu$ and it is denoted by $\mu \uplus \nu$.
The analytic approach to Boolean  convolution is based on the Cauchy transform
\begin{align}
\label{eq:hm:ball}
G_\mu(z)=\int_{\R}\frac{1}{z-y}\,\dd{\mu(y)}
\end{align}
of a probability measure $\mu$. The Cauchy transform is analytic on the upper half
plane $\C^+=\{z |\Im z>0\}$ 
and takes values from the closed lower half plane
 $\C^-=\{z |\Im z<0\}$. 
For measures with compact support, the Cauchy transform is analytic at infinity
and related to the ordinary  moment generating function $M_{\X}$ as follows:
\begin{align}\label{mgf}
M_{\X}(z)=\sum_{n=0}^{\infty}\,\state(\X^n)\,z^n = \frac{1}{z}\,G_\X (1/z).
\end{align}

Moreover, the  moment generating transform 
can be written as
$$
M_\mu(z) = \frac{1}{1-H_\mu(z)}  ,
$$
where $H_\mu(z)$ is analytic in a neighbourhood of zero.
The coefficients of its series expansion
\begin{align} \label{rtr}
H_{\X}(z)=\sum_{n=1}^{\infty}\,K_{n}(\X)z^n
\end{align} 
are called \emph{Boolean  cumulants} of the random variable $\X$. 
Using this, the Boolean convolution can be computed via the identity
\begin{align} \label{Booleanconv}
  H_{\mu\uplus\nu}(z)=H_\mu(z)+H_\nu(z)
  ,
\end{align}
see~\cite{SpeicherWorudi1995}.

\subsection{Boolean infinite divisibility}
\label{ssec:FID}
In analogy with classical probability,
a probability measure $\mu$ on $\R$ is said to be
\emph{Boolean infinitely divisible} (or BID for short)
if for each $n \in \{1, 2, 3, \dots \}$ there exists a probability measure
$\mu_n$ such that
$\mu=
 \mu_n\uplus\mu_n\uplus\dots\uplus\mu_n
$ ($n$-fold Boolean convolution).
Boolean infinite divisibility of a measure $\mu$ is characterized
by the property that its \emph{self-energy transform} $\phi_\mu(z)=zH_{\mu}(1/z)$ has
a Nevanlinna-Pick representation
\cite{SpeicherWorudi1995}
\begin{equation}
 \phi_{\mu}(z)=
  \gamma+\int_{\R}\frac{1+xz}{z-x}\,\dd\rho({x})
\end{equation}
for some $\gamma \in \R$ and some nonnegative finite measure $\rho$.
This  $\rho$ is called the \emph{Boolean L\'evy measure} of $\mu$.
It is worth to mention that all probability measures on $\R$ are Boolean infinitely divisible.

\subsection{Interval Partitions}
We recall some facts about interval partitions. 

Let $S$ be an ordered set. Then $\pi = \{ V_1,\dots, V_p\}$ is a partition of $S$, if the $V_i \neq \emptyset$ are
ordered and disjoint sets $V_i=(v_1,\dots,v_k)$, where $v_1<\dots<v_k$, whose union is $S$. We write $V \in \pi$ if $V$ is a class of $\pi$ and we say that $V$ is a \emph{block of $\pi$}.
 Any partition $\pi$ defines an equivalence relation on $S$,
denoted by $\sim_\pi$, such that the equivalence classes are the blocks $\pi$. 
That is, $i\sim_\pi j$ if $i$ and $j$ belong to the same block of $\pi$.
A block $V$  of a partition $\pi$ is called an \emph{interval} block if $V$ is of the form $V=(k, k+1,\dots, k+l)$ for $k\geq 1$ and $0 \leq l \leq n -k$. 
 A block of $\pi$ is called a \emph{singleton} if it consists of one element.  
 Let $\Sing(\pi)$   denote the set of all singletons of $\pi$.
 
A partition $\pi$ is called \emph{interval partition} if its every block is an interval.
The set of interval partitions of $S$ is denoted by $\NC(S)$,
in the case where $S=[n]:=\{1, \dots , n\}$ we write
$\NC(n):=\NC([n])$. 
$\NC(n)$ is a lattice under refinement order, where we say $\pi\leq \rho$ if
every block of $\pi$ is contained in a block of $\rho$.

The maximal element of $\NC(n)$ under this order is the partition consisting
of only one block and it is denoted by  $\hat{1}_{n}$.
On the other hand, the minimal element $\hat{0}_n$ 
is the unique partition whose every block is a singleton.
Sometimes it is convenient to visualize partitions as diagrams, for example
$\hat{1}_n=
\begin{picture}(25,6.5)(1,0)
  \put(1,0){\line(0,1){7.5}}
  \put(6,0){\line(0,1){7.5}}
  \put(8,0){$\cdots$}
  \put(25,0){\line(0,1){7.5}}
  \put(1,7.5){\line(1,0){24}}
\end{picture}
$
,  $\hat{0}_n=
\begin{picture}(26,3.5)(1,0)
  \put(2,0){\line(0,1){4.5}}
  \put(8,0){\line(0,1){4.5}}
\put(10,0){$\cdots$}
  \put(26,0){\line(0,1){4.5}}
  \put(2,4.5){\line(1,0){0}}
  \put(8,4.5){\line(1,0){0}}
  \put(14,4.5){\line(1,0){0}}
  \put(20,4.5){\line(1,0){0}}
  \put(26,4.5){\line(1,0){0}}
\end{picture}$. Two specific  partitions will play a particularly important role,
namely the \emph{standard matching}
$\onetwo{n}=
\begin{picture}(44,3.5)(1,0)
  \put(2,0){\line(0,1){4.5}}
  \put(8,0){\line(0,1){4.5}}
  \put(14,0){\line(0,1){4.5}}
  \put(20,0){\line(0,1){4.5}}
\put(21.5,0){$\cdots$}
  \put(38,0){\line(0,1){4.5}}
  \put(44,0){\line(0,1){4.5}}
  \put(2,4.5){\line(1,0){6}}
  \put(14,4.5){\line(1,0){6}}
  \put(38,4.5){\line(1,0){6}}
\end{picture}\in\NC(2n)$ 
and  $$
\begin{picture}(56,6.5)(1,0)
  \put(2,0){\line(0,1){4.5}}
  \put(8,0){\line(0,1){4.5}}
  \put(14,0){\line(0,1){4.5}}
  \put(20,0){\line(0,1){4.5}}
  \put(26,0){\line(0,1){4.5}}
  \put(44,0){\line(0,1){4.5}}
  \put(50,0){\line(0,1){4.5}}
  \put(56,0){\line(0,1){4.5}}
\put(28.0,0){$\cdots$}
  \put(8,4.5){\line(1,0){6}}
  \put(20,4.5){\line(1,0){6}}
  \put(44,4.5){\line(1,0){6}}
\end{picture}=\{(1),(2,3),(4,5),\dots, (2n-2,2n-1) ,(2n)\}\in\NC(2n).$$
A partition $\pi$ is called \emph{noncrossing} if different blocks do not interlace, i.e., there is no quadruple of elements $i<j<k<l$ such that $i\sim_\pi k$ and $j\sim_\pi l$  but $i\not\sim_\pi j$. 
The set of non-crossing partitions of $n$ is denoted by $NC(n)$.

Given an interval partition $\pi$ of $\{1,2,\dots,n\}$,
an \emph{interval complement $\Krewl\pi$  of $\pi\in \NC(n)$}  is the maximal
noncrossing partition of the ordered set $\{\bar{1},\bar{2},\dots,\bar{n}\}$
such that $\pi\cup\Krewl\pi$ is a noncrossing partition of
the interlaced set
$\{\bar{1},1,\bar{2},2,\dots,\bar{n},n\}$. 
 The set of interval complement partitions of $\NC(n)$ is denoted by $\ICop(n)$ and the map $\Krewl {} : \NC(n) \to \ICop(n)$ is called the complement map. 
 Let us observe that   $i\sim_{\pi^c} k$, then $1\sim_{\pi^c} i$  and $1\sim_{\pi^c} k$, which 
means that only the first block may be bigger than one  and all other blocks are singletons; see  Figure \ref{fig:examplepartition}.

For $\pi^c \in \ICop(n) $ we denote by $\overline {\pi^{c}}$ its noncrossing
closure, that is the largest noncrossing partition such that  
this  partition is
obtained from $\pi^c$ by taking unions of all singletons 
except the first one 
in the graphical representation, as
in the example shown in Figure \ref{fig:examplepartition} (a similar definition appeared in \cite[Definition 2.3]{AHLV2015}).
\begin{Rem}
Kreweras \cite{Kreweras:1972} discovered an interesting antiisomorphism
of the noncrossing partition, now called the \emph{Kreweras complementation map}. The above definition of interval complement   was adopted from Kreweras \cite{Kreweras:1972}.   
\end{Rem}
\begin{figure}[h]
\begin{equation*}
  \begin{array}{ccccc}
\begin{picture}(32,6.5)(1,0)
   \put(0,8){\line(0,1){4.5}}
  \put(6,8){\line(0,1){4.5}}
  \put(12,8){\line(0,1){4.5}}
  \put(18,8){\line(0,1){4.5}}
  \put(24,8){\line(0,1){4.5}}
  \put(30,8){\line(0,1){4.5}}
  \put(18,12.5){\line(1,0){12}}
  \put(0,12.5){\line(1,0){12}}
\end{picture} &    \to & 
\begin{picture}(32,6.5)(1,0)
   \put(0,8){\line(0,1){8}}
  \put(6,8){\line(0,1){4.5}}
  \put(12,8){\line(0,1){4.5}}
  \put(18,8){\line(0,1){8}}
  \put(24,8){\line(0,1){4.5}}
  \put(30,8){\line(0,1){4.5}}
  \put(0,16){\line(1,0){18}}
\end{picture}
 & \to &
\begin{picture}(32,6.5)(1,0)
   \put(0,8){\line(0,1){8}}
  \put(6,8){\line(0,1){4.5}}
  \put(12,8){\line(0,1){4.5}}
  \put(18,8){\line(0,1){8}}
  \put(24,8){\line(0,1){4.5}}
  \put(30,8){\line(0,1){4.5}}
  \put(0,16){\line(1,0){18}}
   \put(6,12.5){\line(1,0){6}}
    \put(24,12.5){\line(1,0){6}}
\end{picture}
\\
   \substack{\{(1,2,3),(4,5,6)\}\\ \text{interval}}  &&  \substack{ \{(1,4),(2),(3),(5),(6)\}\\ \text{complement} } 
   && \substack{\{(1,4),(2,3),(5,6)\}\\ \text{closure}} 
  \end{array}
\end{equation*}

\begin{equation*}
  \begin{array}{ccccc}
\begin{picture}(32,6.5)(1,0)
   \put(0,8){\line(0,1){4.5}}
  \put(6,8){\line(0,1){4.5}}
  \put(12,8){\line(0,1){4.5}}
  \put(18,8){\line(0,1){4.5}}
  \put(24,8){\line(0,1){4.5}}
  \put(30,8){\line(0,1){4.5}}
  \put(0,12.5){\line(1,0){30}}
\end{picture} &    \to & 
\begin{picture}(32,6.5)(1,0)
   \put(0,8){\line(0,1){4.5}}
  \put(6,8){\line(0,1){4.5}}
  \put(12,8){\line(0,1){4.5}}
  \put(18,8){\line(0,1){4.5}}
  \put(24,8){\line(0,1){4.5}}
  \put(30,8){\line(0,1){4.5}}
\end{picture}
 & \to &
\begin{picture}(32,6.5)(1,0)
   \put(0,8){\line(0,1){4.5}}
  \put(6,8){\line(0,1){4.5}}
  \put(12,8){\line(0,1){4.5}}
  \put(18,8){\line(0,1){4.5}}
  \put(24,8){\line(0,1){4.5}}
  \put(30,8){\line(0,1){4.5}}
  \put(6,12.5){\line(1,0){24}}
\end{picture}
\\
   \substack{\{(1,2,3,4,5,6)\}\\ \text{interval}}  &&  \substack{ \{(1),(2),(3),(4),(5),(6)\}\\ \text{complement} } 
   && \substack{\{(1),(2,3,4,5,6)\}\\ \text{closure}} 
  \end{array}
\end{equation*}
\caption{Examples of Kreweras complementation and closure partitions}
\label{fig:examplepartition}
\end{figure}
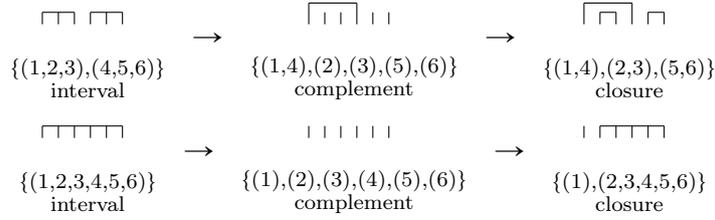

\subsection{Boolean cumulants}
\label{ssec:Booleancumulants}
Given  a  noncommutative  probability space $(\A,\state)$,
 the \emph{Boolean   cumulants} are multilinear functionals $K_n : \mathcal{A}^n\to\C$ 
defined  implicitly in terms of the mixed moments by the relation
\begin{align}
\state(\X_{1}\X_{2}\dots \X_{n}) = \sum_{\pi \in \NC(n)}K_{\pi}(\X_{1},\X_{2},\dots ,\X_{n}),\label{eq:DefinicjaKumulant}
\end{align}
where 
\begin{align}
K_{\pi}(\X_{1},\X_{2},\dots ,\X_{n}):=\Pi_{B \in \pi}K_{\abs{B}}(\X_{i}:i \in B). \label{eq:DefinicjaProduktuKumulant}
\end{align}
 Sometimes we will abbreviate the univariate cumulants as $K_{n}(\X)=K_{n}(\X,\dots ,\X )$.

Boolean cumulants provide a powerful technical tool to investigate
Boolean random variables. 
This is due to the basic property of \emph{vanishing of mixed cumulants}. 
By this we mean the property that 
$$
K_n(X_1,X_2,\dots,X_n)=0
$$
for any family of random variables $X_1,X_2,\dots,X_n$ which can be partitioned
into two mutually Boolean nontrivial subsets. 
For Boolean sequences this can be reformulated as follows.
Let $(X_i)_{i\in\N}$ be a sequence of Boolean random variables and let
$h:[r]\to\N$ be a map. We denote by $\ker h$ the set partition
which is induced by the equivalence relation 
$$
k\sim_{\ker h} l
\ 
\iff
\ 
h(k)=h(l)
.
$$
Similarly, for a multiindex $\underline{i}=i_1i_2\dots i_r$ we denote its kernel $\ker \underline{i}$ by the relation $k \sim l$ if $i_k=i_l$.

In this notation, vanishing of mixed cumulants implies that
\begin{equation}
  \label{eq:kerh>=pi}
  K_\pi(X_{h(1)},X_{h(2)},\dots,X_{h(r)})=0
  \text{ unless $\ker h\geq \pi$.}
\end{equation}

Our main technical tool is the Boolean version,
due to Lehner~\cite{Lehner:2004},   of the classical formula of 
James and Leonov/Shiryaev \cite{James:1958,LeonovShiryaev:1959} which
expresses the cumulants of products in terms of individual cumulants.
\begin{theo}
\label{thm:krawczyk}
Let
$r,n \in \N$ and $ i_1 < i_2 < \dots < i_r = n$ be given and let
$$\rho=\{(1,2,\dots,i_1),(i_1+1,i_1+2,\dots,i_2),\dots,(i_{r-1}+1,i_{r-1}+2,\dots,i_r)\}\in \NC(n)$$ 
be the induced interval partition.
Consider now random variables $\X_1,\dots,\X_n\in\A$.
Then the Boolean cumulants of the products can be expanded
as follows:
\begin{align} 
\label{twr:produktargumentow}
K_r(\X_1\dots \X_{i_1},\dots,\X_{i_{r-1}+1}\dots\X_n)=\sum_{\substack{\pi\in \NC(n) \\ \pi\vee\rho=\hat{1}_{n}} }K_\pi({\X_1,\dots,\X_n}). 
\end{align} 
\end{theo}

Next, \cite[Proposition 3.3]{Popa2009} investigates the properties of Boolean cumulants with
scalars among their entries. 
This result is unexpected, because the identity operator is not Boolean independent of another operator.  
\begin{prop} \label{niezaleznoscPopa}
If $n,m\geq 1$ and $X\in \mathcal{A}^n$, $Y\in \mathcal{A}^m$, then 
\begin{enumerate}
    \item $K_{m+1}(1,Y)=0$,
    \item $K_{m+1}(X,1)=0$,
    \item $K_{n+m+1}(X,1,Y)=K_{n+m}(X,Y)$.
\end{enumerate}

\end{prop}

\subsection{Some probability distributions}
Let us now recall basic properties of some specific probability distributions
which play prominent roles in the present paper.
\subsubsection{Boolean Gaussian distribution}
 A non-commutative random variable $X$ is said to be  \emph{Boolean Gaussian or  Boolean normal }  if $K_r(X)=0$ for $r>2$.
 The reason for the latter is the fact that its Boolean cumulants $K_r(X)=0$ for $r>2$
and it appears in the Boolean version of the central limit theorem.
 The Boolean Gaussian  law with mean zero and variance $a^2$ has distribution
$$\frac{1}{2}\delta_{-a}+\frac{1}{2}\delta_{a}. $$
 Its 
Cauchy-Stieltjes transform
is given by the formula
$$
G_{\mu}(z)=\frac{1}{1-a^2/z} .
$$

For the purpose of this article we say that a family $X_i $ $i\in [n]$ is a \emph{Boolean standard normal family} if $K_1(X_i)=K_2(X_i)=1$, $K_r(X_i)=0$ for $r>2$ and  $X_i $ are Boolean independent.

\begin{Rem}
(1). We assume that the Boolean standard normal family have mean one, because as one can see later, the cumulants of quadratic forms are zero whenewer mean is zero. 

(2). In contrast to the classical convolution, it is not true that, for arbitrary $a\in \R$, the convolution $\mu \uplus \delta_a$ is equal to the shift of measure $\mu$ by the amount $a$.
 This fact is a consequence of Proposition \eqref{niezaleznoscPopa}.    For example, one has
$$
(\frac{1}{2}\delta_{-a}+\frac{1}{2}\delta_{a})\uplus \delta_c= \frac{1}{2}\left(\Big(1+\frac{c}{\sqrt{4a^2+c^2}}\Big)\delta_{(c+\sqrt{4a^2+c^2})/2}+\Big(1-\frac{c}{\sqrt{4a^2+c^2}}\Big)\delta_{(c-\sqrt{4a^2+c^2})/2} \right).
$$
\end{Rem}
\subsubsection{Boolean Poisson distribution}
A non-commutative random variable $\X$ is said to be Boolean Poisson variable if it has  distribution $\nu=\nu(\lambda,\alpha)$ defined by the formula
\begin{align} \label{MPdist}
\nu=\frac{1}{\lambda+1}(\delta_0+ \lambda\delta_{\alpha(\lambda+1)}),
\end{align}
where $\alpha,\lambda\ge 0$. 
The parameters $\lambda$ and $\alpha$ are called the rate and the jump size, respectively.
It is easy to see that if $X$ is Boolean Poisson, $\nu(\lambda,\alpha)$, then $K_n(\X)=\alpha^n\lambda$. Therefore its
$H$-transform has the form
$$H(z)=\frac{\lambda\alpha z}{1-\alpha z}.$$

\subsubsection{Even Poisson distribution}
 We call an element $\X\in \mathcal{A}$ \emph{even Poisson distribution} 
 if its
 even Boolean  cumulants are equal $K_{2n}(\X)=1. $ 
The basic example of such a law is Boolean Poisson distribution with $\lambda=\alpha=1$.

\subsubsection{Identically distributed random variables}
In this article, by identically distributed  random variables we mean random variables with the same cumulant, which are uniformly bounded in the sample size.  
Precisely, we  call that the elements $X_1,\dots,X_n\in \A$ are \emph{identically distributed } if  their all cumulants   are  equal $K_r(\X_1)=\dots =K_r(\X_n)$ for all $r\in\N$ (they can depend on $n$), and 
$\abs{K_r(\X_i)}< C_r$, where $C_r$ is a universal  constant independent of $n$.  This definition play an important role in the section about the limit theorems  of quadratic forms.

\subsection{Special matrices and lemmas }
Let $M_n(\C)$ and $M_n^{sa}(\C)$ denote the set of a scalar and  self-adjoint  matrices.  
For scalars $a,b, c\in\C$  we denote by  $\left[\begin{smallmatrix}
    c &a\\
    b& c
    \end{smallmatrix}\right]_n
\in M_n(\C)
$
the matrix whose diagonal elements are equal to $c$, whose  upper-triangular entries are equal to $a$ and
whose lower-triangular elements are equal to $b$, respectively. 
  For simplicity of notation, we use the  letter $J_n$, $\P$ and $B_n$ to denote the matrices $$J_n=\left[\begin{smallmatrix}
    1 &1\\
    1& 1
    \end{smallmatrix}\right]_n, \text{ } P_n=\frac{1}{n}J_n \text{ and }B_n=i \frac{1}{n}\left[\begin{smallmatrix}
    0 &1\\
    -1 &0
    \end{smallmatrix}\right]_n.$$ 
    In mathematics, a stochastic matrix is a square matrix used to describe the transitions of a Markov chain. Each of its entries is a nonnegative real number representing a probability, i.e., each row summing to $1$. For our purpose, we  introduce  a \emph{zero sum matrix} which is a square matrix in $M_n(\C)$, with each row summing to $0$.    For our next result we will use the following lemmas about matrices.
    
\begin{lemm} \label{lemm:macierzediagonalne}
 Let 
   $A=[a_{i,j}]_{i,j=1}^n\in M_n^{sa}(\C)$  and $\Lambda=\diag(A)=\diag(a_{1,1},\dots,a_{n,n})$ be its diagonal matrix. If  $$\Tr(J_n \Lambda^k A \Lambda^k)=0 \text{ for all even } k\in \N, $$
then  $A$ is a constant diagonal matrix (= multiple of an identity matrix). 
\end{lemm}
\begin{proof}
By direct calculation we have 
$$\Tr(J_n \Lambda^k A \Lambda^k)=\sum_{i,j}(a_{i,i} a_{j,j})^k a_{i,j}=0. $$
The proof is by contradiction, so assume that $A$ is not a constant diagonal matrix.  
Let $C=\max_{i}\abs{a_{i,i}}$ and $max_\Lambda=\{i|i\in [n] \text{ and }\abs{a_{i,i}}=C\}$ be the set of indexes of maximal diagonal elements, then 
$$\sum_{i,j}(a_{i,i} a_{j,j}/C^2)^k a_{i,j}=0.$$
Next, let $k$ go to infinity, then in limit, we
get    $$\sum_{i\in max_\Lambda}a_{i,i}=\#max_\Lambda \times  C= 0,$$ which implies $C=0$ for all $i\in [n]$, which  contradicts the assumption. 
\end{proof}
\begin{lemm}\label{lemma:equvalenttwoforall} Let    $A=[a_{i,j}]_{i,j=1}^n\in M_n^{sa}(\C)$, then 
 $$ \Tr(\J A^{k})=0 \text{ for all $k\in \N$} \Longleftrightarrow  \Tr(\J A^{2})=0 \Longleftrightarrow A\text{ is zero sum matrix}.$$ 
 \end{lemm}
 \begin{proof}We will first prove the equivalence of first two conditions;  then we show equivalence  of last two conditions. 
We use the notation 
\begin{itemize}
\item $\{\lambda_1,\dots,\lambda_n\}$ is the spectrum of $A$, the eigenvalues are real;
\item $\{U_1,\dots,U_n\}$ is the orthonormal basis (or  unitary basis) of corresponding eigenvectors;
\item if $x$ is  an $n$-vector it will be  convenient to denote by $\sigma(x)$ the sum of the coordinates of $x$.
\end{itemize}  
Let $\mathbf{1}$ denote the  vector  $(1,1,\dots,1)$, 
then 
 \begin{align*}
\Tr(\J A^{k})&=\mathbf{1} A^k \mathbf{1}^\ast 
\\&=( \sigma(U_1),\dots, \sigma(U_n)) \diag(\lambda_1^k,\dots,\lambda_n^k)( \sigma(U_1),\dots, \sigma(U_n))^\ast
\\&=\abs {\sigma(U_1)}^2 \lambda_1^k+\dots+ \abs {\sigma(U_n)}^2 \lambda_n^k=0.
\end{align*}
The above equality holds for $k=2$ if and only if  $\sigma(U_i)\lambda_i=0 \Longleftrightarrow \sigma(U_i)=0 \vee \lambda_i=0 $ for all $i\in [n],$ which  means that $\Tr(\J A^{k})=0 $    
 for all $k\in \N$. 
 Now let us observe that
$$\Tr(\J A^{2})=\sum_i \big(\sum_j a_{i,j}{\sum_j  a_{j,i}}\big)=
\sum_i \big(\sum_j a_{i,j}\overline{\sum_j  a_{i,j}}\big)=0 \Longleftrightarrow \sum_j a_{i,j}=0 \text{ for all $i\in[n]$}. $$
\end{proof}
\
\subsection{Convergence in distribution }
\label{sssec:CLT}
In noncommutative probability
we say that a sequence $X_n$ of random variables \emph{converges in
distribution} towards $X$ as $n\to \infty$, denoted by
 $$\X_n
  \xrightarrow{d} X,$$
  if we have for all $m\in\N$
   $$\lim_{n\to \infty}\state(\X_n^m)=\state( X^m)\text{ or equivalently}
   \lim_{n\to \infty}K_m(\X_n)=K_m(X)
   .$$
\subsection{Convergence in state}
In this section we introduce the convergence in state
$\omega$
with density matrix $P_n$, that is, 
\begin{equation*}
\omega(C):=\Tr(P_nC)=\frac{1}{n}\sum_{i,j}^n c_{ij}=\xi^TC\xi, 
\end{equation*}
where, as above, by $\xi$ we denote the unit vector
$\xi=\frac{1}{\sqrt{n}}(1,1,\dots,1)^T$ and $C=[c_{i,j}]_{i,j=1}^n\in M_n(\C)$.
We say that a sequence of $N\times N$ 
deterministic 
matrices $A_N$ have limit distribution $\mu$ with respect to the
state $\omega$ if for every $m\in\IN$ the moments satisfy
$$
\lim_{N\to\infty}\Tr(P_N A_N^m)=\lim_{N\to\infty}\omega(A_N^m)=\int t^{m} \dd{\mu(t)}
.
$$
Note that in this case $\mu$ is not necessarily a probability measure.

\subsection{Combinatorics of tangent numbers}
\label{ssec:tangentnumbers}
The \emph{tangent numbers}
\begin{equation}
  \label{eq:tangentnumbers}
  T_{2k-1}=(-1)^{k+1}\frac{4^{k}(4^{k}-1){B_{2k}}}{2k}
\end{equation}
for $k\in\N $ are the Taylor coefficients of the tangent function
$$
\tan z = \sum_{n=1}^\infty T_n\frac{z^n}{n!} = z + \frac{2}{3!} z^3 +  \frac{16}{5!}z^5 +\frac{272}{7!}z^7 + \dotsm,
$$
see \cite[Page 287]{GrahamKnuthPatashnik:1994}.
The tangent numbers are complemented by the \emph{secant numbers}.  
Together they form
the sequence of \emph{Euler zigzag numbers} $E_n$ 
function
$$
\tan z +\sec z =\sum_{n=0}^\infty
\frac{E_{n}}{n!}z^{n}.
$$ 
These numbers are also called 
\emph{Springer numbers} or
\emph{up-down numbers} \cite{Carlitz:1975:permutations} or
\emph{snake numbers}  \cite{Arnold:1992:snakes,Hoffman:1999:derivative}
and  appear in several different contexts,
see for example \cite{Flajolet1980,Arnold:1991,StanleyVol2,Stanley2010} or  Andr\'e's theorem \cite{Andre1880}.
The \emph{higher order tangent numbers} $T_n^{(k)}$ were introduced by Carlitz and Scoville
\cite{CarlitzScoville:1972} as the coefficients of the Taylor series
$$\tan^{k+1} z =\sum_{n=k+1}^\infty T_n^{(k+1)}\frac{z^n}{n!}.$$
The generating function of the tangent polynomials $T_n(x) = \sum_{k=0}^{n-1}T_n^{(k+1)}x^k$ 
were computed by Carlitz and Scoville's
\cite[Equation (1.6)]{CarlitzScoville:1972} and we have
\begin{equation}
  \label{eq:genfunTnx}
\begin{aligned}
  T(x,z)
  &= \frac{\tan z}{1-x\tan z}\\
  &= \sum_{n=1}^\infty \frac{T_n(x)}{n!} z^n. 
\end{aligned}
\end{equation}

\section{Cumulants of quadratic forms }\label{sec:cumquadrform}
In this subsection we express the cumulants of quadratic forms of  random variables
in terms of the conditional expectations of the system matrix
onto the diagonal matrices. We start with a  lemma, which establishes an isomorphism of interval partition of $r+1$ elements, with special kind of interval partitions of $2r$ points (the Boolean analogue of \cite[Lemma 2.14]{EjsmontLehner:2017}). In some sense the following lemma appears in   literature (see for example \cite{FevrierMastnakNicaSzpojankowski2020}) but no one has systematized it in the following form.

\begin{lemm} 
  \label{lemm:singelton}
  Let $r\in\mathbb{N}$ and $\pi \in \NC(2r)$,
  then $\pi \vee \onetwo{r}=\hat{1}_{2r}$ if and only if $\pi\geq 
\begin{picture}(56,6.5)(1,0)
  \put(2,0){\line(0,1){4.5}}
  \put(8,0){\line(0,1){4.5}}
  \put(14,0){\line(0,1){4.5}}
  \put(20,0){\line(0,1){4.5}}
  \put(26,0){\line(0,1){4.5}}
  \put(44,0){\line(0,1){4.5}}
  \put(50,0){\line(0,1){4.5}}
  \put(56,0){\line(0,1){4.5}}
\put(28.0,0){$\cdots$}
  \put(8,4.5){\line(1,0){6}}
  \put(20,4.5){\line(1,0){6}}
  \put(44,4.5){\line(1,0){6}}
\end{picture}$,
  i.e.,  the elements $2i$ and $2i+1$
   lie in the same block of $\pi$ for $i\in[r-1]$.  
  Consequently
  $$
  \{\pi : \pi \vee \onetwo{r}=\hat{1}_{2r}\}
  = [
\begin{picture}(56,6.5)(1,0)
  \put(2,0){\line(0,1){4.5}}
  \put(8,0){\line(0,1){4.5}}
  \put(14,0){\line(0,1){4.5}}
  \put(20,0){\line(0,1){4.5}}
  \put(26,0){\line(0,1){4.5}}
  \put(44,0){\line(0,1){4.5}}
  \put(50,0){\line(0,1){4.5}}
  \put(56,0){\line(0,1){4.5}}
\put(28.0,0){$\cdots$}
  \put(8,4.5){\line(1,0){6}}
  \put(20,4.5){\line(1,0){6}}
  \put(44,4.5){\line(1,0){6}}
\end{picture}, \hat{1}_{2r}  ] ,
  $$
  is a lattice isomorphic to $\NC(r+1)$.  
\end{lemm}
\begin{cor}\label{wniosekdolematu}
 The above lemma implies that 
the structure of blocks is 
\end{cor}
\setlength\parindent{110pt}$\pi=$ 
\begin{picture}(25,6.5)(1,0)
  \put(1,0){\line(0,1){7.5}}
  \put(6,0){\line(0,1){7.5}}
  \put(8,0){$\cdots$}
  \put(25,0){\line(0,1){7.5}}
  \put(1,7.5){\line(1,0){24}}
\end{picture}
 \text{or } $\pi$ = \begin{picture}(32,6.5)(1,0) 
\put(0,10){\text{Odd}}
\put(60,10){\text{Even inside}}
\put(140,10){\text{Odd}}
 \put(0,-10){\line(0,1){4.5}}
  \put(6,-10){\line(0,1){4.5}}
  \put(12,-10){\line(0,1){4.5}}
  \put(18,-10){\line(0,1){4.5}}
  \put(36,-10){\line(0,1){4.5}}
  \put(42,-10){\line(0,1){4.5}}
\put(20,-10){$\cdots$}
  \put(0,-5.5){\line(1,0){6}}
  \put(12,-5.5){\line(1,0){6}}
  \put(36,-5.5){\line(1,0){6}}
  \put(44,-10){$\cdots$}

 
   \put(0,0){\line(0,1){4.5}}
  \put(6,0){\line(0,1){4.5}}
  \put(12,0){\line(0,1){4.5}}
  \put(18,0){\line(0,1){4.5}}
  \put(36,0){\line(0,1){4.5}}
  \put(0,4.5){\line(1,0){36}}

  \put(60,-10){\line(0,1){4.5}}
  \put(66,-10){\line(0,1){4.5}}
  \put(60,-5.5){\line(1,0){6}}
  \put(68,-10){$\cdots$}
   \put(84,-10){\line(0,1){4.5}}
  \put(90,-10){\line(0,1){4.5}}

   \put(84,-5.5){\line(1,0){6}}
  \put(108,-5.5){\line(1,0){6}}
   \put(92,-10){$\cdots$}
     \put(108,-10){\line(0,1){4.5}}
  \put(114,-10){\line(0,1){4.5}}
\put(116,-10){$\cdots$}
 
   \put(108,0){\line(0,1){4.5}}
  \put(90,0){\line(0,1){4.5}}
  \put(84,0){\line(0,1){4.5}}
  \put(66,0){\line(0,1){4.5}}
  \put(66,4.5){\line(1,0){42}}

  \put(132,-10){\line(0,1){4.5}}
  \put(138,-10){\line(0,1){4.5}}
  \put(132,-5.5){\line(1,0){6}}
  \put(140,-10){$\cdots$}
   \put(156,-10){\line(0,1){4.5}}
  \put(162,-10){\line(0,1){4.5}}
    \put(168,-10){\line(0,1){4.5}}
  \put(174,-10){\line(0,1){4.5}}
   \put(156,-5.5){\line(1,0){6}}
  \put(168,-5.5){\line(1,0){6}}
 
 
   \put(174,0){\line(0,1){4.5}}
  \put(168,0){\line(0,1){4.5}}
  \put(162,0){\line(0,1){4.5}}
  \put(156,0){\line(0,1){4.5}}
  \put(138,0){\line(0,1){4.5}}
  \put(138,4.5){\line(1,0){36}}
  \put(176,-4){$\cdot$}
\end{picture}

\indent
\begin{proof}[Proof of Lemma \ref{lemm:singelton}] 
If there is an odd interval block  inside of the set $\{2,\dots,2r-1\}$, then 
it  splits 
\begin{picture}(44,3.5)(1,0)
  \put(2,0){\line(0,1){4.5}}
  \put(8,0){\line(0,1){4.5}}
  \put(14,0){\line(0,1){4.5}}
  \put(20,0){\line(0,1){4.5}}
\put(21.5,0){$\cdots$}
  \put(38,0){\line(0,1){4.5}}
  \put(44,0){\line(0,1){4.5}}
  \put(2,4.5){\line(1,0){6}}
  \put(14,4.5){\line(1,0){6}}
  \put(38,4.5){\line(1,0){6}}
\end{picture} in two disjoint parts A and B, which cannot be connected by interval partitions; so that cannot give
us $\hat{1}_{2r}$ 
see Figure \ref{fig:splitsingelton} (a), (b). We have a similar
situation, if an even block starts with an odd element, 
see Figure \ref{fig:splitsingelton}  (c). 
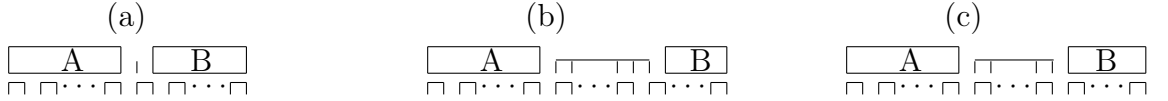
\begin{figure}[h]
\begin{subfigure}[t]{0.3\textwidth}
 
\hspace{7em}  (a)

\vspace{1em} 

 \centering 
 
\begin{picture}(44,3.5)(1,0)
  \put(2,0){\line(0,1){4.5}}
  \put(8,0){\line(0,1){4.5}}
  \put(14,0){\line(0,1){4.5}}
  \put(20,0){\line(0,1){4.5}}
\put(21.5,0){$\cdots$}
  \put(38,0){\line(0,1){4.5}}
  \put(44,0){\line(0,1){4.5}}
  \put(50,0){\line(0,1){4.5}}
  \put(56,0){\line(0,1){4.5}}
    \put(62,0){\line(0,1){4.5}}
  \put(68,0){\line(0,1){4.5}}
  \put(70,0){$\cdots$}
    \put(85,0){\line(0,1){4.5}}
  \put(91,0){\line(0,1){4.5}}
  
  \put(2,4.5){\line(1,0){6}}
  \put(14,4.5){\line(1,0){6}}
  \put(38,4.5){\line(1,0){6}}
 \put(50,4.5){\line(1,0){6}}
  \put(62,4.5){\line(1,0){6}}
  \put(85,4.5){\line(1,0){6}}

     \put(2,8){\line(0,1){10}}
   \put(44,8){\line(0,1){10}}
   \put(2,18){\line(1,0){42}}
   \put(2,8){\line(1,0){42}}
   \put(21.5,9){A}
   
   \put(50,8){\line(0,1){4.5}}
   \put(56,8){\line(0,1){10}}
   \put(91,8){\line(0,1){10}}
   \put(56,18){\line(1,0){35}}
   \put(56,8){\line(1,0){35}}
   \put(70,9){B}
\end{picture}

 \end{subfigure}
~
  \begin{subfigure}[t]{0.3\textwidth}
\hspace{7em}  (b)
 
\vspace{1em}  
 \centering
\begin{picture}(44,3.5)(1,0)
  \put(2,0){\line(0,1){4.5}}
  \put(8,0){\line(0,1){4.5}}
  \put(14,0){\line(0,1){4.5}}
  \put(20,0){\line(0,1){4.5}}
\put(21.5,0){$\cdots$}
  \put(38,0){\line(0,1){4.5}}
  \put(44,0){\line(0,1){4.5}}
  \put(50,0){\line(0,1){4.5}}
  \put(56,0){\line(0,1){4.5}}
  
  \put(2,4.5){\line(1,0){6}}
  \put(14,4.5){\line(1,0){6}}
  \put(38,4.5){\line(1,0){6}}
 \put(50,4.5){\line(1,0){6}}

     \put(2,8){\line(0,1){10}}
   \put(44,8){\line(0,1){10}}
   \put(2,18){\line(1,0){42}}
   \put(2,8){\line(1,0){42}}
   \put(21.5,9){A}
   
   \put(50,8){\line(0,1){4.5}}
   \put(56,8){\line(0,1){4.5}}
   \put(73,8){\line(0,1){4.5}}
   \put(79,8){\line(0,1){4.5}}
    \put(85,8){\line(0,1){4.5}}
   \put(50,13){\line(1,0){35}}
   \put(57.5,0){$\cdots$}
     \put(73,0){\line(0,1){4.5}}
  \put(79,0){\line(0,1){4.5}}
    \put(73,4.5){\line(1,0){6}}
       \put(85,0){\line(0,1){4.5}}
  \put(91,0){\line(0,1){4.5}}
    \put(85,4.5){\line(1,0){6}}
\put(92.5,0){$\cdots$}
    \put(108,0){\line(0,1){4.5}}
  \put(114,0){\line(0,1){4.5}}
      \put(108,4.5){\line(1,0){6}}

        \put(91,8){\line(0,1){10}}
   \put(114,8){\line(0,1){10}}
   \put(91,18){\line(1,0){23}}
   \put(91,8){\line(1,0){23}}
   \put(100,9){B}
\end{picture}

 \end{subfigure}
 ~
  \begin{subfigure}[t]{0.3\textwidth}
\hspace{7em}  (c)
 
\vspace{1em}  
 \centering
\begin{picture}(44,3.5)(1,0)
  \put(2,0){\line(0,1){4.5}}
  \put(8,0){\line(0,1){4.5}}
  \put(14,0){\line(0,1){4.5}}
  \put(20,0){\line(0,1){4.5}}
\put(21.5,0){$\cdots$}
  \put(38,0){\line(0,1){4.5}}
  \put(44,0){\line(0,1){4.5}}
  \put(50,0){\line(0,1){4.5}}
  \put(56,0){\line(0,1){4.5}}
  
  \put(2,4.5){\line(1,0){6}}
  \put(14,4.5){\line(1,0){6}}
  \put(38,4.5){\line(1,0){6}}
 \put(50,4.5){\line(1,0){6}}

     \put(2,8){\line(0,1){10}}
   \put(44,8){\line(0,1){10}}
   \put(2,18){\line(1,0){42}}
   \put(2,8){\line(1,0){42}}
   \put(21.5,9){A}
   
   \put(50,8){\line(0,1){4.5}}
   \put(56,8){\line(0,1){4.5}}
   \put(73,8){\line(0,1){4.5}}
   \put(79,8){\line(0,1){4.5}}
   \put(50,13){\line(1,0){29.5}}
   \put(57.5,0){$\cdots$}
     \put(73,0){\line(0,1){4.5}}
  \put(79,0){\line(0,1){4.5}}
    \put(73,4.5){\line(1,0){6}}
       \put(85,0){\line(0,1){4.5}}
  \put(91,0){\line(0,1){4.5}}
    \put(85,4.5){\line(1,0){6}}
\put(92.5,0){$\cdots$}
    \put(108,0){\line(0,1){4.5}}
  \put(114,0){\line(0,1){4.5}}
      \put(108,4.5){\line(1,0){6}}

        \put(85,8){\line(0,1){10}}
   \put(114,8){\line(0,1){10}}
   \put(85,18){\line(1,0){29}}
   \put(85,8){\line(1,0){29}}
   \put(95,9){B}
\end{picture}

 \end{subfigure}
\caption{Some blocks inside of $\onetwo{n}$}
\label{fig:splitsingelton}
\end{figure}
Hence the condition $\pi \vee \onetwo{r}=\hat{1}_{2r}$
forces $2i$ and $2i + 1$ to lie in the same block. 
This  implies that $\pi$ consists the maximal partition $\pi=
\begin{picture}(25,6.5)(1,0)
  \put(1,0){\line(0,1){7.5}}
  \put(6,0){\line(0,1){7.5}}
  \put(8,0){$\cdots$}
  \put(25,0){\line(0,1){7.5}}
  \put(1,7.5){\line(1,0){24}}
\end{picture}
$ or $\pi$  has 
exactly two odd blocks  on the edges $O_1,O_2$ and  even intervals $E$,  starting with even elements. 
More precisely the blocks have the structure  in Figure  \ref{fig:typeIIbad} or equivalently, $\pi\geq 
\begin{picture}(56,6.5)(1,0)
  \put(2,0){\line(0,1){4.5}}
  \put(8,0){\line(0,1){4.5}}
  \put(14,0){\line(0,1){4.5}}
  \put(20,0){\line(0,1){4.5}}
  \put(26,0){\line(0,1){4.5}}
  \put(44,0){\line(0,1){4.5}}
  \put(50,0){\line(0,1){4.5}}
  \put(56,0){\line(0,1){4.5}}
\put(28.0,0){$\cdots$}
  \put(8,4.5){\line(1,0){6}}
  \put(20,4.5){\line(1,0){6}}
  \put(44,4.5){\line(1,0){6}}
\end{picture}$. 
This proves our claim and hence the assertion is shown.
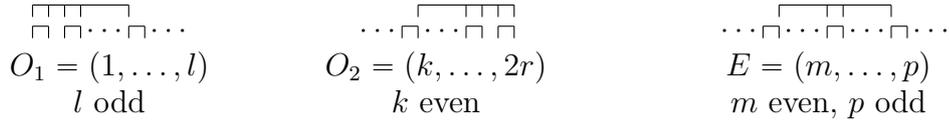
\begin{figure}[h]
\begin{align*}
  \begin{array}{ccccc}
\begin{picture}(32,6.5)(1,0)
  \put(0,0){\line(0,1){4.5}}
  \put(6,0){\line(0,1){4.5}}
  \put(12,0){\line(0,1){4.5}}
  \put(18,0){\line(0,1){4.5}}
  \put(36,0){\line(0,1){4.5}}
  \put(42,0){\line(0,1){4.5}}
\put(20,0){$\cdots$}
  \put(0,4.5){\line(1,0){6}}
  \put(12,4.5){\line(1,0){6}}
  \put(36,4.5){\line(1,0){6}}
  \put(44,0){$\cdots$}

 
   \put(0,8){\line(0,1){4.5}}
  \put(6,8){\line(0,1){4.5}}
  \put(12,8){\line(0,1){4.5}}
  \put(18,8){\line(0,1){4.5}}
  \put(36,8){\line(0,1){4.5}}
  \put(0,12.5){\line(1,0){36}}
\end{picture} &     & 
\begin{picture}(32,3.5)(1,0)
 \put(0,0){$\cdots$}
  \put(16,0){\line(0,1){4.5}}
  \put(22,0){\line(0,1){4.5}}
  \put(16,4.5){\line(1,0){6}}
  \put(24,0){$\cdots$}
   \put(40,0){\line(0,1){4.5}}
  \put(46,0){\line(0,1){4.5}}
    \put(52,0){\line(0,1){4.5}}
  \put(58,0){\line(0,1){4.5}}
   \put(40,4.5){\line(1,0){6}}
  \put(52,4.5){\line(1,0){6}}
 
 
   \put(58,8){\line(0,1){4.5}}
  \put(52,8){\line(0,1){4.5}}
  \put(46,8){\line(0,1){4.5}}
  \put(40,8){\line(0,1){4.5}}
  \put(22,8){\line(0,1){4.5}}
  \put(22,12.5){\line(1,0){36}}
\end{picture}
 & &
\begin{picture}(32,6.5)(1,0) 
 \put(0,0){$\cdots$}
  \put(16,0){\line(0,1){4.5}}
  \put(22,0){\line(0,1){4.5}}
  \put(16,4.5){\line(1,0){6}}
  \put(24,0){$\cdots$}
   \put(40,0){\line(0,1){4.5}}
  \put(46,0){\line(0,1){4.5}}

   \put(64,4.5){\line(1,0){6}}
  \put(40,4.5){\line(1,0){6}}
   \put(48,0){$\cdots$}
     \put(64,0){\line(0,1){4.5}}
  \put(70,0){\line(0,1){4.5}}
\put(72,0){$\cdots$}
 
   \put(64,8){\line(0,1){4.5}}
  \put(46,8){\line(0,1){4.5}}
  \put(40,8){\line(0,1){4.5}}
  \put(22,8){\line(0,1){4.5}}
  \put(22,12.5){\line(1,0){42}}
\end{picture}
\\
   \hspace{2em} O_1=(1,\dots, l) &&  \hspace{2em} O_2=(k,\dots,2r) &&  \hspace{4em}E= (m,\dots,p)
   \\
   \hspace{2em} l \text{ odd} && \hspace{2em} k \text{ even}  &&  \hspace{4em} m \text{ even, } p \text{ odd}
  \end{array}
\end{align*}
\caption{The main structure of blocks which satisfy  $\pi\vee \onetwo{n}=\hat{1}_{2r}$}
\label{fig:typeIIbad}
\end{figure}

\end{proof}

\begin{cor}   \label{eq:poisson}
The Boolean analogue of \cite[Proposition 11.25]{NicaSpeicher:2006}. 
Let  $\hat\pi$ be the image of $\pi\in\NC(r+1)$ under the bijection introduced in Lemma~\ref{lemm:singelton}. 
 
  Let  $X\in \A$, with cumulants
sequence $(\beta_{n})_{n\geq 1}$. Then the cumulants of $X^2$ are given as follows
  $$
    K_r(X^2) = \sum_{ \pi\in \NC(r+1)}\prod_{B\in \hat \pi}\beta_{\abs B}.
$$
In particular if $X$ is a  Boolean Poisson random variable, with  rate $\lambda$ and the jump size $\alpha$, then 
$$
    K_r(X^2) = \alpha^{2r}\sum_{ \pi\in \NC(r+1)}\lambda^{\abs{ \hat{\pi}}}= \alpha^{2r}\sum_{ \pi\in \NC(r+1)}\lambda^{\abs \pi}=\alpha^{2r}\lambda(\lambda+1)^{r},
$$
\i.e. $X^2$ has the Boolean Poisson distribution, with  rate $\lambda$ and the jump size $\alpha^2(\lambda+1)$.

\end{cor}
\noindent
Let $\ED$ be the conditional expectation in  the diagonal algebra
\begin{equation}
  \label{eq:ED}
\begin{aligned}
  \ED:M_n(\mathcal{A})&\to M_n(\C)\\
  A &\mapsto \sum_{i=1}^n E_i\state^{(n)}(A)E_i,
\end{aligned}
\end{equation}
where by $E_i$ we denote the projection matrix onto the $i$-th unit vector
and $\state^{(n)}(A)_{ij}=\state(a_{i,j})$ is the entry-wise state. We also use notation $\ED[ \pi]$, where arguments of $\ED$
 are distributed according to the blocks of partition $\pi$, but the cumulants
are nested inside each other according to the nestings of the blocks of noncrossing $\pi$. For example, if $\pi=\{(1,4),(2,3),(5,6)\}$, then 
$$\ED[ \pi](A_1,\dots,A_6)=\ED(A_1\ED(A_2A_3) A_4)\ED(A_5A_6),$$ 
for more details  see Speicher's \cite{Speicher:1998} theory of operator-valued free probability.

\noindent
The following lemma connects interval partition  with conditional expectation $\ED$  and is the key to
the main result.

\begin{lemm}[\cite{EjsmontLehner:2017}, Lemma 4.2] \label{diagonalnawartoscoczekiwana}
  \label{lem:kreweras}
  For scalar matrices $A\in M_n(\C)$ we have
  \begin{enumerate}[(i)]
   \item \label{it:kreweras1}
    $$
      \sum_{i=1}^n E_iA_1E_iA_2\dotsm E_iA_rE_i
    = \ED(A_1)\ED(A_2)\dotsm \ED(A_r).
    $$
   \item  \label{it:kreweras2}
    Let $\pi\in\NC(r)$, then
    $$
    \sum_{\ker\underline{i}\geq\pi} \ED (A_1E_{i_1}A_2E_{i_2}\dotsm A_rE_{i_r}A_{r+1})
    =\ED[\pi^c] (A_1,A_2,\dots,A_{r+1}).
    $$
  \end{enumerate}
\end{lemm}
\noindent
For $\pi^c \in \ICop(n)$ we define by $\EB[k]:M_n^{ k}(\mathcal{A})\to M_n(\C)$
the corresponding (operator-valued) \emph{Hadamard cumulants} on the closure of $\pi^c$. 
Let $B=(i_1,\dots, i_k)\in \overline {\pi^c}$, then  

\begin{align*}
\begin{aligned}
  \EB[\abs B](A_{i_1}, \dots , A_{i_k}) &= 
  \left\{ \begin{array}{ll}
\ED(A_{i_1} \dots  A_{i_k}) & \textrm{if  $ B=(1,\dots, i_k)$ i.e. contains $1$}\\
\ED(A_{i_1}\odot \dots \odot A_{i_k}) & \textrm{if $B$ doesn't contain $1$ }
\end{array} \right.
\end{aligned}
\end{align*}
where  $\odot$ is the Hadamard product of matrices.
\begin{exam} \label{lemm:condisingelton} 
(1). If $\pi=\{(1,2,3),(4,5,6)\}$, then  $\overline {\pi^c}=\{(1,4),(2,3),(5,6)\}$, and 
$$\EB[\overline {\pi^c}](A_1,\dots,A_6)=\ED(A_1\ED(A_2\odot A_3) A_4)\ED(A_5\odot A_6).$$ 
\noindent
(2). For the singleton partition and scalar matrices $A_i\in M_n(\C)$ we have 
\begin{align*}
\ED[
\begin{picture}(26,3.5)(1,0)
  \put(2,0){\line(0,1){4.5}}
  \put(8,0){\line(0,1){4.5}}
\put(10,0){$\cdots$}
  \put(26,0){\line(0,1){4.5}}
  \put(2,4.5){\line(1,0){0}}
  \put(8,4.5){\line(1,0){0}}
  \put(14,4.5){\line(1,0){0}}
  \put(20,4.5){\line(1,0){0}}
  \put(26,4.5){\line(1,0){0}}
\end{picture}] ({A_1,\dots,A_r})
&=\ED (A_1)\dots \ED ( A_r)
\\&=
\ED ({A_1\odot\dots\odot A_r}).
\end{align*}
\end{exam}
\noindent
We will use the following result as the main technical tool   to express the cumulants of quadratic forms of Boolean
random variables in terms of the diagonal map of matrices.
\begin{prop} \label{prop:CykliczneVariancja}
  Let $\X_1, \X_2,\dots, \X_n\in \A$ be a family of  Boolean independent random variables,
   $A=[a_{i,j}]_{i,j=1}^n\in M_n^{sa}(\C)$ 
  and $T_n=\sum_{i,j} a_{i,j}X_iX_j$ a quadratic form. 
    The cumulants of $T_n$ are given by
  \begin{enumerate}[(i)]
   \item \label{it:cyclic3}
\begin{align}  \label{eq:kumulantsamplevariancenotiid}
&K_r(T_n)=\sum_{i_0, i_1,\dots,i_r\in[n]} 
            \Tr(\J E_{i_0} AE_{i_1}\dots AE_{i_r})\,
            \sum_{\substack{ \pi\in \NC(2r)\\ 
                \pi \vee \onetwo{r}=\hat{1}_{2r}}}
            K_\pi(X_{i_0},X_{i_1},X_{i_1},\dots,X_{i_{r-1}},X_{i_{r-1}},X_{i_r}).
    \end{align}
    \item If we assume in addition that $X_i$ are identically distributed, then 
    \begin{align}  \label{eq:kumulantsamplevariance}
      K_r(T_n)&=\sum_{ \pi\in \NC(r+1)}
     \sum_{\ker\underline{i}\geq \pi}\Tr(\J E_{i_0} AE_{i_1}\dots AE_{i_r})
 K_{\hat\pi}(X),
\intertext{which can be expressed as a  convolution with Hadamard cumulants} 
 &=\sum_{ \pi\in \NC(r+1)}\Tr (\EB[\overline{\Krewl \pi } ] (\underbrace{\J,A,\dots,A}_{r+1}))K_{\hat\pi}(X), \label{eq:kumulantsamplevariance2}
    \end{align} 
 where   $\hat\pi$ is the image of $\pi\in\NC(r+1)$ under the bijection introduced in Lemma~\ref{lemm:singelton}. 
  \end{enumerate}
\end{prop}
\begin{proof}

Let $Z_{i,j}=a_{i,j}X_i X_j $, then from the definition of $T_n$ we see that 
\begin{align*}
 K_r(T_n)
 &= \sum_{i_1,i_2,\dots,i_{2r}\in[n]}
      K_r(Z_{i_{1},i_{2}},
         Z_{i_{3},i_{4}},
         \dots,
         Z_{i_{2r-1},i_{2r}}).
\intertext{
  We can apply Lemma~\ref{lemm:singelton} in the reverse direction and obtain}
 &= \sum_{\substack{
     i_1,i_2,\dots,i_{2r}  \in[n]   \\
     \ker \underline{i}\geq 
\begin{picture}(56,6.5)(1,0)
  \put(2,0){\line(0,1){4.5}}
  \put(8,0){\line(0,1){4.5}}
  \put(14,0){\line(0,1){4.5}}
  \put(20,0){\line(0,1){4.5}}
  \put(26,0){\line(0,1){4.5}}
  \put(44,0){\line(0,1){4.5}}
  \put(50,0){\line(0,1){4.5}}
  \put(56,0){\line(0,1){4.5}}
\put(28.0,0){$\cdots$}
  \put(8,4.5){\line(1,0){6}}
  \put(20,4.5){\line(1,0){6}}
  \put(44,4.5){\line(1,0){6}}
\end{picture} }}
      K_r(Z_{i_{1},i_{2}},
         Z_{i_{3},i_{4}},
         \dots,
         Z_{i_{2r-1},i_{2r}})
\\
 &= \sum_{i_0,i_1,i_2,\dots,i_r\in [n]}
      K_r(Z_{i_0,i_1},Z_{i_1,i_2},Z_{i_2,i_3}\dots ,Z_{i_{r-2},i_{r-1}},Z_{i_{r-1},i_r}).
\intertext{
  We now expand further and obtain}
 &= \sum_{i_0,i_1,i_2,\dots,i_r\in[n]}
      a_{i_0,i_1} a_{i_1,i_2} \dotsm a_{i_{r-1,i_r}}
      K_r(X_{i_0}X_{i_1},X_{i_1}X_{i_2},\dots,X_{i_{r-1}}X_{i_r}).
\intertext{
  Let us remind that $\J=\left[\begin{smallmatrix}
    1 &1\\
    1& 1
    \end{smallmatrix}\right]_n=[b_{i,j}]_{i,j=1}^n$, then }
&= \sum_{i_0,i_1,i_2,\dots,i_r\in[n]}b_{i_r,i_0}
      a_{i_0,i_1} a_{i_1,i_2} \dotsm a_{i_{r-1,i_r}}
      K_r(X_{i_0}X_{i_1},X_{i_1}X_{i_2},\dots,X_{i_{r-1}}X_{i_r})
       \\
 &=\sum_{i_0,i_1,\dots,i_r\in[n]} 
            \Tr(\J E_{i_0} AE_{i_1}AE_{i_2}\dots AE_{i_r})\, K_r(X_{i_0}X_{i_1},X_{i_1}X_{i_2},\dots,X_{i_{r-1}}X_{i_r})
      \\
 &=\sum_{i_0,i_1,\dots,i_r\in[n]} 
            \Tr(\J E_{i_0} AE_{i_1}AE_{i_2}\dots AE_{i_r})\,
            \sum_{\substack{ \pi\in \NC(2r)\\ 
                \pi \vee \onetwo{r}=\hat{1}_{2r}}}
            K_\pi(X_{i_0},X_{i_1},X_{i_1},X_{i_2},\dots,X_{i_{r-1}},X_{i_r}),
 \\
\intertext{which yields \eqref{eq:kumulantsamplevariancenotiid}. Now denoting by $\hat\pi$ the image of $\pi\in\NC(r+1)$ under the bijection introduced in Lemma~\ref{lemm:singelton}, we can rewrite this as}
 &=\sum_{ \pi\in \NC(r+1)}
 \sum_{\ker\underline{i}\geq \pi}\Tr(\J E_{i_0} AE_{i_1}AE_{i_2}\dots AE_{i_r})
 K_{\hat\pi}(X)
 \intertext{which yields  \eqref{eq:kumulantsamplevariance}. Now we use   Lemma \ref{diagonalnawartoscoczekiwana} and obtain  }
 &=\sum_{ \pi\in \NC(r+1)}\Tr (\ED[\Krewl \pi] (\underbrace{\J,A,\dots,A}_{r+1}))
 K_{\hat\pi}(X).
\intertext {We further use the special structure of $\Krewl \pi$, namely that the first block decomposes $\Krewl \pi$ into disjoint  segments,  each consisting of a singleton. On every  segment  we can apply  Remark \ref{lemm:condisingelton} and obtain \eqref{eq:kumulantsamplevariance2}.}
\end{align*}

\end{proof}
\begin{cor}   \label{eq:semisirgular}
1. In the case of the Boolean  normal distribution, with  $K_1=c$, $K_2=1$ and $K_r=0$ for $r\geq 3$, formula 
  \eqref{eq:kumulantsamplevariance} has only one contributing term $\pi=
\begin{picture}(26,3.5)(1,0)
  \put(2,0){\line(0,1){4.5}}
  \put(8,0){\line(0,1){4.5}}
\put(10,0){$\cdots$}
  \put(26,0){\line(0,1){4.5}}
  \put(2,4.5){\line(1,0){0}}
  \put(8,4.5){\line(1,0){0}}
  \put(14,4.5){\line(1,0){0}}
  \put(20,4.5){\line(1,0){0}}
  \put(26,4.5){\line(1,0){0}}
\end{picture}$ and 
  takes the particular  form
  \begin{equation}
    \label{eq:cumQnsemi}
    K_r(T_n) = c^2\Tr(\J A^r).
  \end{equation}
\noindent  
 In fact, we have proved the following slightly more general statement.
  Let $X_i$ be a  Boolean independent standard normal family  of  random variables
  and let $A_1,A_2,\dots,A_r \in M_n(\C)$ be
  arbitrary scalar matrices. Then  the joint cumulant of $T_k=\sum_{i,j}a^{(k)}_{i,j}X_iX_j$ is
  $$
  K_r(T_1,T_2,\dots,T_r)=
      \Tr(\J A_1 A_2 \dots A_r).
  $$

\end{cor}
\section{The Boolean quadratic forms}
\noindent
In this section we consider the Boolean analogue of some   characterization problems (associated with quadratic form) from classical and free probability. Our main technical tool  are the formulas introduced in the  previous paragraph.  
 The first problem is the Boolean analogue of \cite[Proposition 2.2]{HiwatashiKurodaNagisaYoshida:1999a} and \cite[Theorem 3.4]{HiwatashiKurodaNagisaYoshida:1999a} which can be formulated as follows.
 \begin{prop}
Let $\X_1, \X_2,\dots, \X_n\in \A$ be a Boolean standard normal family,  
   $A=[a_{i,j}]_{i,j=1}^n\in M_n^{sa}(\C)$ and $\lambda_1,\dots,\lambda_n$ and $\sigma(U_1),\dots,\sigma(U_n)$ are as in the proof of Lemma \ref{lemma:equvalenttwoforall}.   Then 
\begin{enumerate}
\item  the $H$-transform of the
distribution of the quadratic form 
 can be written in the
form $$H_{T_n}(z)=\sum_{j=1}^n \frac{\abs{\sigma(U_i)}^2\lambda_i z}{1-\lambda_iz} ,$$
\item  the random variable
$T_n$
has  the Boolean Poisson  distribution $\nu(\lambda,\alpha)$ if and only if the matrix $A $ has 
\begin{enumerate}
\item at least one eigenvalue  equal to $\lambda$ and   $$\sum_{i \text{ such that }\lambda_i=\lambda}\abs{\sigma(U_i)}^2=\alpha;$$
\item all eigenvalues different than $0$ and $\lambda$ have corresponding eigenvectors with the sum of the coordinates equal to zero.
\end{enumerate} 
\end{enumerate}
 \end{prop}
\begin{proof}(1)
From  the proof of Lemma \ref{lemma:equvalenttwoforall}, we obtain 
$$H_{T_n}(z)=\sum_{k=1}\Tr(\J A^{k})z^k=\sum_{k=1}({\abs {\sigma(U_1)}}^2 \lambda_1^k+\dots+ \abs {\sigma(U_n)}^2\lambda_n^k)z^k=\sum_{j=1}^n \frac{\abs{\sigma(U_i)}^2\lambda_i z}{1-\lambda_iz}.$$
(2) Again by Lemma  \ref{lemma:equvalenttwoforall}, we have 
 \begin{align}
\Tr(\J A^{k})&=\abs {\sigma(U_1)}^2 \lambda_1^k+\dots+ \abs {\sigma(U_n)}^2 \lambda_n^k=\alpha\lambda^k \nonumber
\intertext{or equivalently}
\Tr(\J (A/\lambda)^{k})&=\abs {\sigma(U_1)}^2 (\lambda_1/\lambda)^k+\dots+ \abs {\sigma(U_n)}^2 (\lambda_n/\lambda)^k=\alpha. \label{rownoscmomentow}
\end{align}
Therefore, if $A$ satisfies    (a) and (b),   then $T_n$
has  the Boolean Poisson   distribution $\nu(\lambda,\alpha)$. 
To   verify the sufficient condition, first we compute the limit of $$\lim_{k\to \infty}\Tr(\J (A/\lambda)^{2k})=\begin{cases}
0& \text{ if }\lambda_i<\lambda \text{ for all $i\in [n]$},\\
\infty& \text{ if }\lambda_i>\lambda \text{ and } \sigma(U_i)>0 \text{ at least for one $i\in [n]$},
\end{cases}$$ which is contradiction with $\alpha>0$.  This  means that if $\lambda_i> \lambda$, then $\sigma(U_i)=0$. Now, without loss of generality, we can assume that $\lambda_i\leq \lambda$.  From the last analysis we also conclude that  $\Lambda:=\{i|i\in [n] \text{ and }\lambda_i=\lambda\}$ is not empty and let $\Lambda^c:=\{i|i\in [n] \text{ and }\lambda_i<\lambda\}$. We thus can rewrite the equation \eqref{rownoscmomentow} as    
$$\sum_{i\in\Lambda^c}\abs {\sigma(U_i)}^2 (\lambda_i/\lambda)^k
=\alpha-\sum_{i\in\Lambda}\abs {\sigma(U_i)}^2.$$
The limit of the left-hand side is zero as $k\to \infty$ and  we find that $\alpha=\sum_{i\in\Lambda}\abs {\sigma(U_i)}^2$. 
 In particular, we prove that
 $$\sum_{i\in\Lambda^c}\abs {\sigma(U_i)}^2 \lambda_i^k
=0 \quad \text{ for $k$ even,}$$
which means that $\sigma(U_i)=0$, whenever   $\lambda_i\neq 0$ for $i\in\Lambda^c. $
\end{proof}
\noindent The Boolean analogue of  \cite[Theorem 2.3]{HiwatashiKurodaNagisaYoshida:1999a} can be formulated as follows.  
\begin{prop} \label{propozycjaniezaleznosc}
 Let $\X_1, \X_2,\dots, \X_n \in \A_{sa}$ be a Boolean standard normal family  and let    $A=[a_{i,j}]_{i,j=1}^n,B=[b_{i,j}]_{i,j=1}^n \in M_n^{sa}(\C)$. 
Then  quadratic forms  $\Q_1=\sum_{i=1}^na_{i,j}\X_i\X_j$ and $\Q_2=\sum_{i=1}^n b_{i,j}\X_i\X_j$  are Boolean independent iff 
\begin{enumerate}
   \item $BA^k$  and  $AB^k$  are zero sums matrices  for all  $k\in \N$ if $n\geq 3$;
    \item $BA$ and  $AB$  are zero sums matrices  if   $n=2$,
\end{enumerate}
\label{twr:HiwatashiKurodaNagisaYoshida} 
\end{prop}

\begin{Rem}
\begin{enumerate}[1.]
    \item Proposition \ref{propozycjaniezaleznosc} implies that there is a gap in \cite[Proposition 4.6]{Lehner:2003}. It is a consequence that  the author does not assume the fact that identity operator is not Boolean independent with the sample $\X_1, \X_2,\dots, \X_n$. 
    \item The following example show that we cannot formulate the part (1) of Proposition \ref{propozycjaniezaleznosc}, without parameter $k$. We can define the matrices $A$ and $B$ for $n \geq 3$ which satisfy $J_nAB=0$, $J_nBA=0$ and $J_nA^kB \neq 0$, $J_nB^kA \neq 0$ for some $k \in \N$.
    
    \noindent
    Computer calculations show that it is possible to find many different such matrices. For example, for $n=3$
    $$A=\left[\begin{smallmatrix}
    -15 & 6 & 1 \\
    6& 9 & -8  \\
    1 & -8 & 5
    \end{smallmatrix}\right]\text{ and } B=\left[\begin{smallmatrix}
    -1 & 16 & 60 \\
    16 & 44 & 90  \\
    60 & 90 & 75
    \end{smallmatrix}\right]. $$ 
Then $J_3AB=0$, $J_3BA=0$ and $J_3A^2B \neq 0$, $J_3B^2A \neq 0$.

\noindent
For $n=4$ we find matrices
    $$A=\left[\begin{smallmatrix}
    1 & 2 & 3 & 4 \\
    2 & -5 & 7 & 16  \\
    3 & 7 & 10 & 10 \\
    4 & 16 & 10 & 10
    \end{smallmatrix}\right]\text{ and } B=\left[\begin{smallmatrix}
    81 & -9 & -9 & -9\\
    -9 & 1 & 1 & 1  \\
    -9 & 1 & 1 & 1 \\
    -9 & 1 & 1 & 1
    \end{smallmatrix}\right]. $$
Then $J_4AB=0$ and $J_4BA=0$. Here $J_4A^{11}B \neq 0$ and $J_4B^{12}A \neq 0$.

\noindent
It is possible to find such matrices $A$ and $B$ for $n \geq 5$. You can create the system of equations using $J_nAB=0$ and $J_nBA=0$ and then find the set of solutions satisfying $J_nA^kB \neq 0$ or $J_nB^kA \neq0$ for some $k \geq 2$.

\end{enumerate}
\end{Rem}

\begin{proof}
(1). 
We can write the joint cumulants of $Q_1$, $Q_2$ as
$$K_{i_1+j_1\dots+i_k+j_k}(\underbrace{Q_1,\dots,Q_1}_{i_1},\underbrace{Q_2,\dots,Q_2}_{j_1},\dots, \underbrace{Q_1,\dots,Q_1}_{i_k},\underbrace{Q_2,\dots,Q_2}_{j_k})=\Tr (\J A^{i_1}B^{j_1}\dots A^{i_k}B^{j_k}),$$
where  $ i_1,j_1,\dots,i_k,j_k  \in \N_0.$ Assume that $Q_1$ and $Q_2$
are Boolean independent, then mixed cumulants vanish and  
 in particular we have  $\Tr (\J AB^kA)=\frac{1}{n}\Tr (\J AB^kA\J)=0$. Then by
the Schwarz inequality for nonnegative self-adjoint operators we have $\J A^kB=0$ (by symmetry we also get $\J B^kA=0$), i.e., $BA^k$ is zero sum matrix. 
For the inverse, let us observe that $BA^k$ and $AB^k$  are zero sums matrices if and only if $\J A^kB=0$ and $\J B^kA=0$, which imply that joint cumulants of $Q_1$ and $Q_2$ disappear.  
\newline
(2). If all mixed cumulants disappear then from part (1) we see that $BA$ and  $AB$  are zero sums matrices. For the converse, let us observe that  if $J_2 AB=0$ and $J_2 BA=0$, then $AB=\left[\begin{smallmatrix}
    a+ib & -a-ib\\
    -a-ib& a +ib\\
    \end{smallmatrix}\right]$ for some $a,b\in \R$, because sums of rows and columns must be zero. Let us observe that  multiplication of two self-adjoint matrices $A$ and 
 $B$ has the structure $\left[\begin{smallmatrix}
    x-iz& \cdot\\
    \cdot & y +iz\\
    \end{smallmatrix}\right]$ for some $x,y,z \in \R$. From this we conclude that  $AB=BA=\left[\begin{smallmatrix}
    a & -a\\
    -a& a\\
    \end{smallmatrix}\right]$, namely the matrices $A$ and $B$ commute. In this case, we can calculate the
    joint cumulants of $Q_1$, $Q_2$ as
\begin{align*}K_{i_1+j_1\dots+i_k+j_k}&(\underbrace{Q_1,\dots,Q_1}_{i_1},\underbrace{Q_2,\dots,Q_2}_{j_1},\dots, \underbrace{Q_1,\dots,Q_1}_{i_k},\underbrace{Q_2,\dots,Q_2}_{j_k})\\&=\Tr (J_2 AB A^{i_1+\dots+i_k-1}B^{j_1+\dots+j_k-1})=0.
\end{align*}
 \end{proof}
 \noindent
 Kagan and Shalaevski \cite{KaganShalaevski} have shown that if the random variables  $\X_1,\dots,\X_n$ are i.i.d.  and the  distribution of $\sum_{i=1}^n(\X_i+a_i)^2$,   $a_i\in \mathbb{R}$ depends only on $\sum_{i=1}^na_i^2$ , then each $\X_i\sim N(0, \sigma)$. 
In the next theorem, we will show that  the Boolean version of this theorem is true. Since $(\X_i+a_i)^2$, $i\in\{1,\dots,n\}$ are not Boolean  independent (by Proposition \ref{niezaleznoscPopa}), we cannot use the standard arguments as in  \cite[Theorem 3.2]{HiwatashiKurodaNagisaYoshida:1999a}
or in \cite{Ejsmont2016}.

\begin{theo}
 Let $\X_1, \X_2,\dots, \X_n$ be Boolean  independent   identically distributed copies of a
random variable $X$ with mean $0$ and variance $1$. Then the sums $P=\sum_{i=1}^n(\X_i+a_i)^2$ have the law that
depends on $(a_1,\dots, a_n)$ through $\sum_{i=1}^na_i^2$  only  for every $a_i \in \mathbb{R}$ if and only if $X$ is Boolean Gaussian  random variable.
\end{theo}

\begin{proof}
Suppose that $X_i$ have the Boolean   normal distribution with mean zero and variance one, then by Proposition \ref{niezaleznoscPopa} we have
             \begin{align}
K_r(\X_{i_1}+a_{i_1},\X_{i_2}+a_{i_2},\dots,\X_{i_r}+a_{i_r}) &=
\begin{cases}
0& \text{ for }i_1\neq i_r,\\
a_{i_2}\dots a_{i_{r-1}} K_2(\X_{i_1},\X_{i_1})& \text{ for } i_1= i_r,
\end{cases}\nonumber
\\&= 
\begin{cases}
0& \text{ for }i_1\neq i_r,\\
a_{i_2}\dots a_{i_{r-1}} & \text{ for } i_1= i_r.
\end{cases} \label{cumulantyfacznadlagausowskich}
\end{align}
First we apply the decomposition from Corollary \ref{wniosekdolematu} and obtain 
 \begin{align*}
    K_r(P) &=
               \sum_{i_1,\dots,i_{r}}\sum_{\substack{\pi\in \NC(2r)\\ \pi\vee \onetwo{r}=\hat{1}_{2r}}}
               K_\pi(\X_{i_1}+a_{i_1},\X_{i_1}+a_{i_1},                
                    \dots,        \X_{i_{r}}+a_{i_{r}},\X_{i_{r}}+a_{i_{r}})
                    \\
&=\sum_{i_1,\dots,i_{r}}\sum_{\substack{\pi=
\begin{picture}(25,6.5)(1,0)
  \put(1,0){\line(0,1){7.5}}
  \put(6,0){\line(0,1){7.5}}
  \put(8,0){$\cdots$}
  \put(25,0){\line(0,1){7.5}}
  \put(1,7.5){\line(1,0){24}}
\end{picture}
}}
K_\pi(\X_{i_1}+a_{i_1},\X_{i_1}+a_{i_1},                
                    \dots,        \X_{i_{r}}+a_{i_{r}},\X_{i_{r}}+a_{i_{r}})
                    \\
&+\sum_{i_1,\dots,i_{r}}\sum_{\substack{\pi\in \NC(2r) \\ \pi=\begin{picture}(32,6.5)(-8,0) 
 

 
   \put(0,0){\line(0,1){4.5}}
  \put(6,0){\line(0,1){4.5}}
  \put(12,0){\line(0,1){4.5}}
  \put(18,0){\line(0,1){4.5}}
  \put(36,0){\line(0,1){4.5}}
  \put(0,4.5){\line(1,0){36}}

 \put(45,-2){$\cdots$}

\put(116,-2){$\cdots$}
 
   \put(108,0){\line(0,1){4.5}}
  \put(90,0){\line(0,1){4.5}}
  \put(84,0){\line(0,1){4.5}}
  \put(66,0){\line(0,1){4.5}}
  \put(66,4.5){\line(1,0){42}}

 
   \put(174,0){\line(0,1){4.5}}
  \put(168,0){\line(0,1){4.5}}
  \put(162,0){\line(0,1){4.5}}
  \put(156,0){\line(0,1){4.5}}
  \put(138,0){\line(0,1){4.5}}
  \put(138,4.5){\line(1,0){36}}
\end{picture}}}
K_\pi(\X_{i_1}+a_{i_1},\X_{i_1}+a_{i_1},                
                    \dots,        \X_{i_{r}}+a_{i_{r}},\X_{i_{r}}+a_{i_{r}})
\intertext{using equation \eqref {cumulantyfacznadlagausowskich}, we  eliminate the zero contribution indexes and  obtain}                  
&=\sum_{i_1,\dots,i_{r-1}}\sum_{\substack{\pi=
\begin{picture}(25,6.5)(1,0)
  \put(1,0){\line(0,1){7.5}}
  \put(6,0){\line(0,1){7.5}}
  \put(8,0){$\cdots$}
  \put(25,0){\line(0,1){7.5}}
  \put(1,7.5){\line(1,0){24}}
\end{picture}
}}
K_\pi(\X_{i_1},\X_{i_1}+a_{i_1},                
                    \dots,        \X_{i_{1}}+a_{i_{1}},\X_{i_{1}})
                    \\
&+\sum_{i_1,\dots,i_{r}}\sum_{\substack{\pi\in \NC(2r) 
\\ \ker\underline{i}\geq \begin{picture}(32,12)(-4,0) 
 

 
   \put(0,0){\line(0,1){9}}
  \put(6,0){\line(0,1){9}}
  \put(12,0){\line(0,1){4.5}}
  \put(18,0){\line(0,1){4.5}}
  \put(36,0){\line(0,1){9}}
  \put(0,9){\line(1,0){174}} 
  \put(12,4.5){\line(1,0){6}} 
 
 \put(45,-2){$\cdots$}

\put(116,-2){$\cdots$}
 
   \put(108,0){\line(0,1){9}}
  \put(90,0){\line(0,1){4.5}}
  \put(84,0){\line(0,1){4.5}}
  \put(66,0){\line(0,1){9}}
  \put(84,4.5){\line(1,0){6}}

 
   \put(174,0){\line(0,1){9}}
  \put(168,0){\line(0,1){9}}
  \put(162,0){\line(0,1){4.5}}
  \put(156,0){\line(0,1){4.5}}
  \put(138,0){\line(0,1){9}}
  \put(156,4.5){\line(1,0){6}}
\end{picture}
\\ \begin{picture}(32,6.5)(-16,0) 
 

 
   \put(0,0){\line(0,1){4.5}}
  \put(6,0){\line(0,1){4.5}}
  \put(12,0){\line(0,1){4.5}}
  \put(18,0){\line(0,1){4.5}}
  \put(36,0){\line(0,1){4.5}}
  \put(0,4.5){\line(1,0){36}}

 \put(45,-2){$\cdots$}

\put(116,-2){$\cdots$}
 
   \put(108,0){\line(0,1){4.5}}
  \put(90,0){\line(0,1){4.5}}
  \put(84,0){\line(0,1){4.5}}
  \put(66,0){\line(0,1){4.5}}
  \put(66,4.5){\line(1,0){42}}

 
   \put(174,0){\line(0,1){4.5}}
  \put(168,0){\line(0,1){4.5}}
  \put(162,0){\line(0,1){4.5}}
  \put(156,0){\line(0,1){4.5}}
  \put(138,0){\line(0,1){4.5}}
  \put(138,4.5){\line(1,0){36}}
\end{picture}}}
K_\pi(\X_{i_1}+a_{i_1},\X_{i_1}+a_{i_1},                
                    \dots,        \X_{i_{r}}+a_{i_{r}},\X_{i_{r}}+a_{i_{r}}).
                     \end{align*}
The first sums correspond to $\sum_{i_1,\dots,i_{r-1}}a_{i_1}^2\dots a_{i_{r-1}}^2=\big(\sum_{i=1}a_{i}^2\big)^{r-1}$.   The second summation can be decomposed according to the number of singletons. Indeed,  the outer block (which connects two odd blocks)    has the contribution  $\sum_{i}a_{i}^2$ and the remaining indexes run over all independent set of   pairs. It is not difficult to see that the number of these pairs depends on the number of  singletons $r-1+\#\Sing(\pi)-\#\pi$ and the corresponding contribution can be described as below.
 \begin{center}
\begin{tabular}{c c c c  c}
   & $\ker\underline{i}\geq$  &\begin{picture}(178,12)(-4,0) 
 

 
   \put(0,0){\line(0,1){9}}
  \put(6,0){\line(0,1){9}}
  \put(12,0){\line(0,1){4.5}}
  \put(18,0){\line(0,1){4.5}}
  \put(36,0){\line(0,1){9}}
  \put(0,9){\line(1,0){174}} 
  \put(12,4.5){\line(1,0){6}} 
 
 \put(45,-2){$\cdots$}

\put(116,-2){$\cdots$}
 
   \put(108,0){\line(0,1){9}}
  \put(90,0){\line(0,1){4.5}}
  \put(84,0){\line(0,1){4.5}}
  \put(66,0){\line(0,1){9}}
  \put(84,4.5){\line(1,0){6}}

 
   \put(174,0){\line(0,1){9}}
  \put(168,0){\line(0,1){9}}
  \put(162,0){\line(0,1){4.5}}
  \put(156,0){\line(0,1){4.5}}
  \put(138,0){\line(0,1){9}}
  \put(156,4.5){\line(1,0){6}}
\end{picture} & $\leadsto$ & $\big(\sum_{i=1}a_{i}^2\big)^{r+\#\Sing(\pi)-\#\pi}$
\\Outer & $\ker\underline{i}\geq$ &\begin{picture}(178,12)(-4,0) 
 

 
   \put(0,0){\line(0,1){9}}
  \put(6,0){\line(0,1){9}}
  
  \put(36,0){\line(0,1){9}}
  \put(0,9){\line(1,0){174}} 
 
 \put(45,-2){$\cdots$}

\put(116,-2){$\cdots$}
 
   \put(108,0){\line(0,1){9}}
  \put(66,0){\line(0,1){9}}

 
   \put(174,0){\line(0,1){9}}
  \put(168,0){\line(0,1){9}}

  \put(138,0){\line(0,1){9}}
\end{picture} & $\leadsto$ & $\sum_{i}a_{i}^2$ 
\\ Reminder pairs & 
$\ker\underline{i}\geq$ &\begin{picture}(178,12)(-4,0) 
 
 
  \put(12,0){\line(0,1){4.5}}
  \put(18,0){\line(0,1){4.5}}

  \put(12,4.5){\line(1,0){6}} 
 
 \put(45,-2){$\cdots$}

\put(116,-2){$\cdots$}

  \put(90,0){\line(0,1){4.5}}
  \put(84,0){\line(0,1){4.5}}

  \put(84,4.5){\line(1,0){6}}


  \put(162,0){\line(0,1){4.5}}
  \put(156,0){\line(0,1){4.5}}

  \put(156,4.5){\line(1,0){6}}
\end{picture} & $\leadsto$   & $\big(\sum_{i=1}a_{i}^2\big)^{r-1+\#\Sing(\pi)-\#\pi}$
\end{tabular}
\end{center}
 More precisely, it can be written as the sums
 \begin{align*}
&=\big(\sum_{i=1}a_{i}^2\big)^{r-1}                  \\
&+\sum_{\substack{\pi\in \NC(2r) \\ \pi\vee \onetwo{r}=\hat{1}_{2r}
\\ 
\Sing(\pi)=\{(1),(2r)\}}}\sum_{i_1,\dots,i_{r}}
K_\pi(\X_{i_1}+a_{i_1},\X_{i_1}+a_{i_1},                
                    \dots,        \X_{i_{r}}+a_{i_{r}},\X_{i_{r}}+a_{i_{r}})
                                        \\
&+\sum_{\substack{\pi\in \NC(2r) \\ \pi\vee \onetwo{r}=\hat{1}_{2r}
\\ 
\Sing(\pi)=\{(1)\} \vee \Sing(\pi)=\{(2r)\}}}\sum_{i_1,\dots,i_{r}}
K_\pi(\X_{i_1}+a_{i_1},\X_{i_1}+a_{i_1},                
                    \dots,        \X_{i_{r}}+a_{i_{r}},\X_{i_{r}}+a_{i_{r}})
                                          \\
&+\sum_{\substack{\pi\in \NC(2r) \\ \pi\vee \onetwo{r}=\hat{1}_{2r}
\\ 
\Sing(\pi)=\emptyset\\
\pi\neq
\begin{picture}(25,6.5)(1,0)
  \put(1,0){\line(0,1){7.5}}
  \put(6,0){\line(0,1){7.5}}
  \put(8,0){$\cdots$}
  \put(25,0){\line(0,1){7.5}}
  \put(1,7.5){\line(1,0){24}}
\end{picture}
}}\sum_{i_1,\dots,i_{r}}
K_\pi(\X_{i_1}+a_{i_1},\X_{i_1}+a_{i_1},                
                    \dots,        \X_{i_{r}}+a_{i_{r}},\X_{i_{r}}+a_{i_{r}}) 
                    \end{align*}
Finally, if we take into account the corresponding contribution of singletons, this  can be represented in a more compact form:
                    \begin{align*}          &=\big(\sum_{i=1}a_{i}^2\big)^{r-1}             \\&+\sum_{\substack{\pi\in \NC(2r) \\ \pi\vee \onetwo{r}=\hat{1}_{2r}
\\ 
\Sing(\pi)=\{(1),(2r)\}}}
\big(\sum_{i=1}a_{i}^2\big)^{r+2-\#\pi}
                                        \\
&+\sum_{\substack{\pi\in \NC(2r) \\ \pi\vee \onetwo{r}=\hat{1}_{2r}
\\ 
\Sing(\pi)=\{(1)\} \vee \Sing(\pi)=\{(2r)\}}}
\big(\sum_{i=1}a_{i}^2\big)^{r+1-\#\pi}
                                        \\& +\sum_{\substack{\pi\in \NC(2r) \\ \pi\vee \onetwo{r}=\hat{1}_{2r}
\\ 
\Sing(\pi)=\emptyset\\
\pi\neq
\begin{picture}(25,6.5)(1,0)
  \put(1,0){\line(0,1){7.5}}
  \put(6,0){\line(0,1){7.5}}
  \put(8,0){$\cdots$}
  \put(25,0){\line(0,1){7.5}}
  \put(1,7.5){\line(1,0){24}}
\end{picture}
}}\big(\sum_{i=1}a_{i}^2\big)^{r-\#\pi}
\\&=\sum_{\substack{\pi\in \NC(2r) \\ \pi\vee \onetwo{r}=\hat{1}_{2r}
}}
\big(\sum_{i=1}a_{i}^2\big)^{r+\#\Sing(\pi)-\#\pi}.
  \end{align*}
  Therefore, the distribution of $P$ depends only on $\sum_{i=1}a_{i}^2$. 
  To prove the reverse implication, we consider the following random
variable
$$Y(a)=(\X_1+a)^2+\X_2^2+\dots\X_n^2.$$
 Let us observe that   by Proposition \ref{niezaleznoscPopa} we have
             \begin{align}
K_m(\X+a) =
\sum_{i=0}^{m-2}{{m-2}\choose{i}}a^iK_{m-i}(X).
\label{rownaniexplusa}
 \end{align}
Direct calculations show that $$K_2(Y(a))=nK_{4}(X)+4aK_{3}(X)+4a^2.$$  By the hypothesis, the right hand side of $K_4(Y(a))$ does not depend on $a^2$, and thus we obtain $K_3(X)=0$.
Next we evaluate the $r$-th cumulants of $Y(a)$, i.e, 
\begin{align*}
    K_r(Y(a)) &=
            \sum_{\substack{\pi\in \NC(2r)\\ \pi\vee \onetwo{r}=\hat{1}_{2r}}}
               K_\pi(Y(a))
                  \\
&=K_{
\begin{picture}(25,6.5)(1,0)
  \put(1,0){\line(0,1){7.5}}
  \put(6,0){\line(0,1){7.5}}
  \put(8,0){$\cdots$}
  \put(25,0){\line(0,1){7.5}}
  \put(1,7.5){\line(1,0){24}}
\end{picture}
}(X_1+a)+ K_{
\begin{picture}(25,6.5)(1,0)
  \put(1,0){\line(0,1){7.5}}
  \put(6,0){\line(0,1){7.5}}
  \put(8,0){$\cdots$}
  \put(25,0){\line(0,1){7.5}}
  \put(1,7.5){\line(1,0){24}}
\end{picture}
}(X_2)+\dots+K_{
\begin{picture}(25,6.5)(1,0)
  \put(1,0){\line(0,1){7.5}}
  \put(6,0){\line(0,1){7.5}}
  \put(8,0){$\cdots$}
  \put(25,0){\line(0,1){7.5}}
  \put(1,7.5){\line(1,0){24}}
\end{picture}
}(X_{n}) + K_{
\begin{picture}(24,6.5)(1,0)
  \put(1,0){\line(0,1){7.5}}
  \put(6,0){\line(0,1){7.5}}
  \put(8,0){$\cdots$}
  \put(24,0){\line(0,1){7.5}}
  \put(6,7.5){\line(1,0){18}}
\end{picture}
}(X_1+a) +K_{\begin{picture}(24,6.5)(1,0)
  \put(1,0){\line(0,1){7.5}}
  \put(18,0){\line(0,1){7.5}}
  \put(2,0){$\cdots$}
  \put(24,0){\line(0,1){7.5}}
  \put(1,7.5){\line(1,0){18}}
\end{picture}}(X_1+a) 
                    \\
&+\sum_{\substack{\pi\in \NC(2r) \\ 
 \pi\vee \onetwo{r}=\hat{1}_{2r}
\\ \pi\neq 
\begin{picture}(25,6.5)(1,0)
  \put(1,0){\line(0,1){7.5}}
  \put(6,0){\line(0,1){7.5}}
  \put(8,0){$\cdots$}
  \put(25,0){\line(0,1){7.5}}
  \put(1,7.5){\line(1,0){24}}
\end{picture}

\wedge \pi \neq  
\begin{picture}(24,6.5)(1,0)
  \put(1,0){\line(0,1){7.5}}
  \put(6,0){\line(0,1){7.5}}
  \put(8,0){$\cdots$}
  \put(24,0){\line(0,1){7.5}}
  \put(6,7.5){\line(1,0){18}}
\end{picture}
 
\wedge \pi \neq  \begin{picture}(24,6.5)(1,0)
  \put(1,0){\line(0,1){7.5}}
  \put(18,0){\line(0,1){7.5}}
  \put(2,0){$\cdots$}
  \put(24,0){\line(0,1){7.5}}
  \put(1,7.5){\line(1,0){18}}
\end{picture} 
}}
K_\pi(Y(a))
                    \\
&=K_{2r}(X+a)+ (n-1)K_{2r}(X)+2K_{
\begin{picture}(24,6.5)(1,0)
  \put(1,0){\line(0,1){7.5}}
  \put(6,0){\line(0,1){7.5}}
  \put(8,0){$\cdots$}
  \put(24,0){\line(0,1){7.5}}
  \put(6,7.5){\line(1,0){18}}
\end{picture}
}(X+a) 
                    \\
&+\sum_{\substack{\pi\in \NC(2r) \\ 
 \pi\vee \onetwo{r}=\hat{1}_{2r}
\\ \pi\neq 
\begin{picture}(25,6.5)(1,0)
  \put(1,0){\line(0,1){7.5}}
  \put(6,0){\line(0,1){7.5}}
  \put(8,0){$\cdots$}
  \put(25,0){\line(0,1){7.5}}
  \put(1,7.5){\line(1,0){24}}
\end{picture}

\wedge \pi \neq  
\begin{picture}(24,6.5)(1,0)
  \put(1,0){\line(0,1){7.5}}
  \put(6,0){\line(0,1){7.5}}
  \put(8,0){$\cdots$}
  \put(24,0){\line(0,1){7.5}}
  \put(6,7.5){\line(1,0){18}}
\end{picture}
 
\wedge \pi \neq  \begin{picture}(24,6.5)(1,0)
  \put(1,0){\line(0,1){7.5}}
  \put(18,0){\line(0,1){7.5}}
  \put(2,0){$\cdots$}
  \put(24,0){\line(0,1){7.5}}
  \put(1,7.5){\line(1,0){18}}
\end{picture} 
}}
K_\pi(Y(a)).
\intertext{This in turn, by \eqref{rownaniexplusa},  is equal to 
}
&=\sum_{i=0}^{2r-2}{{2r-2}\choose{i}}a^iK_{2r-i}(X)+ (n-1)K_{2r}(X)+2\sum_{i=0}^{2r-3}{{2r-3}\choose{i}}a^{i+1}K_{2r-i-1}(X)  
                    \\
&+\sum_{\substack{\pi\in \NC(2r) \\ 
 \pi\vee \onetwo{r}=\hat{1}_{2r}
\\ \pi\neq 
\begin{picture}(25,6.5)(1,0)
  \put(1,0){\line(0,1){7.5}}
  \put(6,0){\line(0,1){7.5}}
  \put(8,0){$\cdots$}
  \put(25,0){\line(0,1){7.5}}
  \put(1,7.5){\line(1,0){24}}
\end{picture}

\wedge \pi \neq  
\begin{picture}(24,6.5)(1,0)
  \put(1,0){\line(0,1){7.5}}
  \put(6,0){\line(0,1){7.5}}
  \put(8,0){$\cdots$}
  \put(24,0){\line(0,1){7.5}}
  \put(6,7.5){\line(1,0){18}}
\end{picture}
 
\wedge \pi \neq  \begin{picture}(24,6.5)(1,0)
  \put(1,0){\line(0,1){7.5}}
  \put(18,0){\line(0,1){7.5}}
  \put(2,0){$\cdots$}
  \put(24,0){\line(0,1){7.5}}
  \put(1,7.5){\line(1,0){18}}
\end{picture} 
}}
K_\pi(Y(a))
                    \\
&=nK_{2r}(X)+2arK_{2r-1}(X)+\sum_{i=2}^{2r-2}{{2r-2}\choose{i}}a^iK_{2r-i}(X) +2\sum_{i=1}^{2r-3}{{2r-3}\choose{i}}a^{i+1}K_{2r-i-1}(X)  
                    \\
&+\sum_{\substack{\pi\in \NC(2r) \\ 
 \pi\vee \onetwo{r}=\hat{1}_{2r}
\\ \pi\neq 
\begin{picture}(25,6.5)(1,0)
  \put(1,0){\line(0,1){7.5}}
  \put(6,0){\line(0,1){7.5}}
  \put(8,0){$\cdots$}
  \put(25,0){\line(0,1){7.5}}
  \put(1,7.5){\line(1,0){24}}
\end{picture}

\wedge \pi \neq  
\begin{picture}(24,6.5)(1,0)
  \put(1,0){\line(0,1){7.5}}
  \put(6,0){\line(0,1){7.5}}
  \put(8,0){$\cdots$}
  \put(24,0){\line(0,1){7.5}}
  \put(6,7.5){\line(1,0){18}}
\end{picture}
 
\wedge \pi \neq  \begin{picture}(24,6.5)(1,0)
  \put(1,0){\line(0,1){7.5}}
  \put(18,0){\line(0,1){7.5}}
  \put(2,0){$\cdots$}
  \put(24,0){\line(0,1){7.5}}
  \put(1,7.5){\line(1,0){18}}
\end{picture} 
}}
K_\pi(Y(a))
      \\
&=nK_{2r}(X)+2arK_{2r-1}(X)+p_{2r}(a),
                    \end{align*}
where $p_{2r}(a)=\sum_{i=2}^{2r}c_ia^i$  
 is the polynomial of indeterminate $a$  such that the coefficients $c_i$ are a polynomial of $K_1,\dots,K_{2r-2}$. We would like to emphasize that there is no linear term in $p_{2r}(a)$ because we know that $K_3(X)=0$. Since $K_r(Y(a))=K_r(Y(-a))$  for all
$a\in \R$, we have  $K_{i}(X)=0$ for odd $i\geq 5$.  Further we use induction over $r$ 
to prove that all even cumulants of order at least four disappear. 
Indeed, we consider the following random
variable
$$\tilde Y(a,b)=(\X_1+a)^2+(\X_2+b)^2+\X_3^2+\dots+\X_n^2.$$
First we consider the special case $K_{4}(\tilde Y(a,b))$, which we present without  explanation  how we got the coefficients, because the
details will be explained  later in the general case. We start with  $r=4$ and obtain
\begin{align*}
    K_{4}(\tilde Y(a,b))&=nK_{8}(X)+(a^2+b^2)(26+2(n-2))K_{6}(X)\\&+\Big(4(a^4+b^4)+(2+(n-2))(a^2+b^2)^2+45(a^2+b^2)\Big)K_{4}(X)\\&+      \sum_{\substack{\pi\in \NC(8) \\ \pi\vee \onetwo{r}=\hat{1}_{8}
}}
\big(\sum_{i=1}a_{i}^2\big)^{4+\#\Sing(\pi)-\#\pi}.
\end{align*}  
By assumption we have 
$K_4(\tilde Y(a,0))=K_4(\tilde Y(a/\sqrt{2},a/\sqrt{2}))$ for all $a\in \R$. 
Comparing
the coefficients of the term $a^4$, we  get $K_4(X)=0.$ 
 Let us assume that $K_4(X)=\dots=K_{2r-8}=0$.  We extract from  $K_{r}(\tilde Y(a,b))$ for $r\geq 5$ all the factors involving $K_{2r-4}(X)$. 
Now we would like to present the  partitions and corresponding coefficients  which contribute to the term  $K_{2r-4}(X)$.
\begin{center}
\begin{tabular}{ c c c }
Partition $\pi$ & $\leadsto$ & coefficient associated with $K_{2r-4}(X)$ in expansion of $K_{\pi}(X)$ \\ 
\begin{picture}(25,6.5)(1,0)
  \put(1,0){\line(0,1){7.5}}
  \put(6,0){\line(0,1){7.5}}
  \put(8,0){$\cdots$}
  \put(25,0){\line(0,1){7.5}}
  \put(1,7.5){\line(1,0){24}}
\end{picture}
 & $\leadsto$ & $\big({{2r-2}\choose{2}}+2(r-2)\big)(a^2+b^2)+{{r-2}\choose{2}}(n-2)(a^2+b^2)^2$ \\ 
\begin{picture}(24,6.5)(1,0)
  \put(1,0){\line(0,1){7.5}}
  \put(6,0){\line(0,1){7.5}}
  \put(8,0){$\cdots$}
  \put(24,0){\line(0,1){7.5}}
  \put(6,7.5){\line(1,0){18}}
\end{picture}
 
$+$  \begin{picture}(24,6.5)(1,0)
  \put(1,0){\line(0,1){7.5}}
  \put(18,0){\line(0,1){7.5}}
  \put(2,0){$\cdots$}
  \put(24,0){\line(0,1){7.5}}
  \put(1,7.5){\line(1,0){18}}
\end{picture} & $\leadsto$&  $2(2r-3)(a^2+b^2)$ \\  
\begin{picture}(42,6.5)(1,0)
  \put(1,0){\line(0,1){7.5}}
  \put(6,0){\line(0,1){7.5}}
  \put(8,0){$\cdots$}
  \put(24,0){\line(0,1){7.5}}
  \put(6,7.5){\line(1,0){18}}
    \put(30,0){\line(0,1){7.5}}
  \put(36,0){\line(0,1){7.5}}
    \put(42,0){\line(0,1){7.5}}
    \put(30,7.5){\line(1,0){12}}
\end{picture}$+$ 
\begin{picture}(42,6.5)(1,0)
  \put(0,0){\line(0,1){7.5}}
  \put(6,0){\line(0,1){7.5}}
  \put(12,0){\line(0,1){7.5}}
  \put(18,0){\line(0,1){7.5}}
  \put(20,0){$\cdots$}
  \put(18,7.5){\line(1,0){18}}
  \put(36,0){\line(0,1){7.5}}
    \put(42,0){\line(0,1){7.5}}
    \put(0,7.5){\line(1,0){12}}
\end{picture} & $\leadsto$ & $2(a^2+b^2)$ 
\\  
\begin{picture}(42,6.5)(1,0)
  \put(1,0){\line(0,1){7.5}}
  \put(6,0){\line(0,1){7.5}}
  \put(8,0){$\cdots$}
  \put(24,0){\line(0,1){7.5}}
  \put(6,7.5){\line(1,0){18}}
    \put(30,0){\line(0,1){7.5}}
  \put(36,0){\line(0,1){7.5}}
    \put(42,0){\line(0,1){7.5}}
    \put(30,7.5){\line(1,0){6}}
\end{picture}$+$ \begin{picture}(42,6.5)(1,0)
  \put(0,0){\line(0,1){7.5}}
  \put(6,0){\line(0,1){7.5}}
  \put(12,0){\line(0,1){7.5}}
  \put(18,0){\line(0,1){7.5}}
  \put(20,0){$\cdots$}
  \put(18,7.5){\line(1,0){18}}
  \put(36,0){\line(0,1){7.5}}
    \put(42,0){\line(0,1){7.5}}
    \put(6,7.5){\line(1,0){6}}
\end{picture} & $\leadsto$ & $2(a^2+b^2)$  
\\  
\begin{picture}(36,6.5)(1,0)
  \put(0,0){\line(0,1){7.5}}
  \put(6,0){\line(0,1){7.5}}
  \put(12,0){\line(0,1){7.5}}
  \put(18,0){\line(0,1){7.5}}
  \put(20,0){$\cdots$}
  \put(18,7.5){\line(1,0){18}}
  \put(36,0){\line(0,1){7.5}}
    \put(0,7.5){\line(1,0){12}}
\end{picture}$+$ \begin{picture}(42,6.5)(1,0)
  \put(0,0){\line(0,1){7.5}}
  \put(2,0){$\cdots$}
  \put(18,0){\line(0,1){7.5}}
  \put(0,7.5){\line(1,0){18}}
    \put(24,0){\line(0,1){7.5}}
  \put(30,0){\line(0,1){7.5}}
    \put(36,0){\line(0,1){7.5}}
    \put(24,7.5){\line(1,0){12}}
\end{picture} & $\leadsto$ & $2(2r-5)(a^2+b^2)$
\\  
\begin{picture}(36,6.5)(1,0)
  \put(0,0){\line(0,1){7.5}}
  \put(6,0){\line(0,1){7.5}}
  \put(12,0){\line(0,1){7.5}}
  \put(18,0){\line(0,1){7.5}}
  \put(20,0){$\cdots$}
  \put(18,7.5){\line(1,0){18}}
  \put(36,0){\line(0,1){7.5}}
    \put(6,7.5){\line(1,0){6}}
\end{picture}$+$ \begin{picture}(42,6.5)(1,0)
  \put(0,0){\line(0,1){7.5}}
  \put(2,0){$\cdots$}
  \put(18,0){\line(0,1){7.5}}
  \put(0,7.5){\line(1,0){18}}
    \put(24,0){\line(0,1){7.5}}
  \put(30,0){\line(0,1){7.5}}
    \put(36,0){\line(0,1){7.5}}
    \put(24,7.5){\line(1,0){6}}
\end{picture} & $\leadsto$ & $2(2r-5)(a^2+b^2)$
\\  
 \begin{picture}(30,6.5)(1,0)
  \put(1,0){\line(0,1){7.5}}
  \put(6,0){\line(0,1){7.5}}
  \put(8,0){$\cdots$}
  \put(24,0){\line(0,1){7.5}}
  \put(6,7.5){\line(1,0){18}}
    \put(30,0){\line(0,1){7.5}}

\end{picture} & $\leadsto$ & ${{2r-4}\choose{2}}(a^4+b^4)+2(r-2)a^2b^2$  
\end{tabular}
\end{center}
Essential to understanding how we obtain these different coefficients is to understand the last term, which we explain in detail. Let us observe that \begin{align*}K_{ \begin{picture}(30,6.5)(1,0)
  \put(1,0){\line(0,1){7.5}}
  \put(6,0){\line(0,1){7.5}}
  \put(8,0){$\cdots$}
  \put(24,0){\line(0,1){7.5}}
  \put(6,7.5){\line(1,0){18}}
    \put(30,0){\line(0,1){7.5}}

\end{picture}}(\tilde Y(a,b))=&K_{ \begin{picture}(30,6.5)(1,0)
  \put(1,0){\line(0,1){7.5}}
  \put(6,0){\line(0,1){7.5}}
  \put(8,0){$\cdots$}
  \put(24,0){\line(0,1){7.5}}
  \put(6,7.5){\line(1,0){18}}
    \put(30,0){\line(0,1){7.5}}

\end{picture}}((\X_1+a)^2)+K_{ \begin{picture}(30,6.5)(1,0)
  \put(1,0){\line(0,1){7.5}}
  \put(6,0){\line(0,1){7.5}}
  \put(8,0){$\cdots$}
  \put(24,0){\line(0,1){7.5}}
  \put(6,7.5){\line(1,0){18}}
    \put(30,0){\line(0,1){7.5}}

\end{picture}}((\X_2+b)^2)\\&+K_{ \begin{picture}(30,6.5)(1,0)
  \put(1,0){\line(0,1){7.5}}
  \put(6,0){\line(0,1){7.5}}
  \put(8,0){$\cdots$}
  \put(24,0){\line(0,1){7.5}}
  \put(6,7.5){\line(1,0){18}}
    \put(30,0){\line(0,1){7.5}}

\end{picture}}((\X_1+a)^2,(\X_2+b)^2,(\X_1+a)^2,\dots,(\X_1+a)^2)+\dots\\&+ K_{ \begin{picture}(30,6.5)(1,0)
  \put(1,0){\line(0,1){7.5}}
  \put(6,0){\line(0,1){7.5}}
  \put(8,0){$\cdots$}
  \put(24,0){\line(0,1){7.5}}
  \put(6,7.5){\line(1,0){18}}
    \put(30,0){\line(0,1){7.5}}

\end{picture}}((\X_1+a)^2,(\X_1+a)^2,\dots,(\X_2+b)^2,(\X_1+a)^2)
\\&+K_{ \begin{picture}(30,6.5)(1,0)
  \put(1,0){\line(0,1){7.5}}
  \put(6,0){\line(0,1){7.5}}
  \put(8,0){$\cdots$}
  \put(24,0){\line(0,1){7.5}}
  \put(6,7.5){\line(1,0){18}}
    \put(30,0){\line(0,1){7.5}}

\end{picture}}((\X_2+b)^2,(\X_1+a)^2,(\X_2+b)^2,\dots,(\X_2+b)^2)+\dots\\&+ K_{ \begin{picture}(30,6.5)(1,0)
  \put(1,0){\line(0,1){7.5}}
  \put(6,0){\line(0,1){7.5}}
  \put(8,0){$\cdots$}
  \put(24,0){\line(0,1){7.5}}
  \put(6,7.5){\line(1,0){18}}
    \put(30,0){\line(0,1){7.5}}

\end{picture}}((\X_2+b)^2,(\X_2+b)^2,\dots,(\X_1+a)^2,(\X_2+b)^2)
\intertext{by Equation \eqref{rownaniexplusa}, we have }
=&a^2\sum_{i=0}^{2r-4}{{2r-4}\choose{i}}a^iK_{2r-2-i}(X)+b^2\sum_{i=0}^{2r-4}{{2r-4}\choose{i}}b^iK_{2r-2-i}(X)
\\&+(r-2)a^2b^2\sum_{i=0}^{2r-6}{{2r-6}\choose{i}}a^iK_{2r-4-i}(X)+(r-2)a^2b^2\sum_{i=0}^{2r-6}{{2r-6}\choose{i}}b^iK_{2r-4-i}(X).
\end{align*}
From above we obtain coefficient expansions for $K_{2r-4}(X)$ associated with partition \begin{picture}(30,6.5)(1,0)
  \put(1,0){\line(0,1){7.5}}
  \put(6,0){\line(0,1){7.5}}
  \put(8,0){$\cdots$}
  \put(24,0){\line(0,1){7.5}}
  \put(6,7.5){\line(1,0){18}}
    \put(30,0){\line(0,1){7.5}}

\end{picture} namely $${{2r-4}\choose{2}}(a^4+b^4)+2(r-2)a^2b^2=(r-2)\big((2r-6)(a^4+b^4)+((a^2+b^2)^2)\big).$$ 
 Finally, we get the general formula
\begin{align*}
    K_{r}(\tilde Y(a,b))=&nK_{2r}(X)+(a^2+b^2)\Big(1+{{2r-2}\choose{2}}+2{{2r-3}\choose{1}}+(r-2)(n-2)\Big)K_{2r-2}(X)\\&+\Big[(a^2+b^2)\big(14r-26+(r-1)(2r-3)\big)\\&+(r-2)(2r-6)(a^4+b^4)+\Big(r-2+{{r-2}\choose{2}}(n-2)\Big)(a^2+b^2)^2\Big]K_{2r-4}(X)\\&+\sum_{\substack{\pi\in \NC(2r) \\ \pi\vee \onetwo{r}=\hat{1}_{2r}
}}
\big(\sum_{i=1}a_{i}^2\big)^{r+\#\Sing(\pi)-\#\pi},
    \end{align*} 
where the last factor is obtained by using  the induction assumption and the fact  that odd cumulants disappear, namely $K_{2r-5}(X)=K_{2r-6}(X)=\dots K_3(X)=0$. More precisely, we obtain the last term by using  the argument from the beginning of this proof, because for the rest of the calculation we can assume that $X_i$ have the Boolean normal distribution with mean zero and variance one. 
By assumption, we have 
$K_r(\tilde Y(a,0))=K_r(\tilde Y(a/\sqrt{2},a/\sqrt{2}))$ for all $a\in \R$, 
which yields $K_{2r-4}(X)=0.$ \end{proof}

\subsection{The  Boolean cancellation phenomenon}

In \cite[Lemma 2.17]{EjsmontLehner:2017} we established a curious cancellation result for  symmetrized squares
of centered linear statistics that their distributions  do not depend on the odd cumulants. 
  Now we show that one can use the decomposition from Lemma \ref{lemm:singelton} for a much simpler
approach  leading to a  similar result in the Boolean case, and moreover, we  show that only one cumulant contributes to the distribution.

\begin{prop}
  \label{lem:oddcancellation}
  Let $\X_1, \X_2,\dots, \X_n$ be Boolean independent identically distributed copies
  of a random variable $X$ 
  and
  $L=\sum_{i=1}^n\alpha_i X_i$ be a linear form with $\state(L)=0$. 
  Then the $r^{th}$ cumulant of the quadratic statistic
  $$
  P=\sum_{\sigma\in \SG_n} L_\sigma^2,
  $$
has only one contributing  
term, namely the maximal partition $\hat{1}_{2r}=
\begin{picture}(25,6.5)(1,0)
  \put(1,0){\line(0,1){7.5}}
  \put(6,0){\line(0,1){7.5}}
  \put(8,0){$\cdots$}
  \put(25,0){\line(0,1){7.5}}
  \put(1,7.5){\line(1,0){24}}
\end{picture}
$. 
\end{prop}

\begin{proof}
Let $B_k=(1,\dots,2k-1)$, with  $k\in [r]$ be  an enumeration of the first  odd block of $[2r]$ and let 
$\NC^{B_k}(2r)=\{\pi\in\NC(2r)\mid B_k\in\pi \}$. 
Directly from Lemma \ref{lemm:singelton}, we have the disjoint decomposition 
\begin{align}\label{eq:Ktilde}
\NC(2r)= \NC^{B_1}(2r) \cup \NC^{B_2}(2r)\cup \NC^{B_3}(2r) \cup\dots \cup \NC^{B_r}(2r)\cup \{
\begin{picture}(25,6.5)(1,0)
  \put(1,0){\line(0,1){7.5}}
  \put(6,0){\line(0,1){7.5}}
  \put(8,0){$\cdots$}
  \put(25,0){\line(0,1){7.5}}
  \put(1,7.5){\line(1,0){24}}
\end{picture}
\}.
\end{align} 
First, we apply the product formula of Theorem~\ref{thm:krawczyk} and obtain
  \begin{align*}
    K_r(P) &= \sum_{\sigma_1,\dots,\sigma_r\in\SG_n}
               K_r(L_{\sigma_1}^2,L_{\sigma_2}^2,\dots,L_{\sigma_r}^2)\\
           &= \sum_{\sigma_1,\dots,\sigma_r\in\SG_n}
               \sum_{\substack{\pi\in \NC(2r)\\ \pi\vee \onetwo{r}=\hat{1}_{2r}}}
               K_\pi(L_{\sigma_1},L_{\sigma_1},
                    L_{\sigma_2},L_{\sigma_2},
                    \dots,
                    L_{\sigma_r},L_{\sigma_r})\\
           &= \sum_{\substack{\pi\in \NC(2r)\\ \pi\vee \onetwo{r}=\hat{1}_{2r}}}
              \tilde{K}_\pi(L)
  \end{align*}
  where $
  \tilde{K}_\pi(L)
  =\sum_{\sigma_1,\sigma_2,\dots,\sigma_r\in\SG_n} 
    K_\pi(L_{\sigma_1},L_{\sigma_1},
          L_{\sigma_2},L_{\sigma_2},
          \dots,
          L_{\sigma_r},L_{\sigma_r}).
  $
 We can split  the last sums according to the decomposition  \eqref{eq:Ktilde} 
  as 
  \begin{equation}
    \label{eq:telescope}
    \sum_{\substack{\pi\in \NC(2r)\\ \pi\vee \onetwo{r}=\hat{1}_{2r}}}
              \tilde{K}_\pi(L) = 
    \sum_{\pi=
\begin{picture}(25,6.5)(1,0)
  \put(1,0){\line(0,1){7.5}}
  \put(6,0){\line(0,1){7.5}}
  \put(8,0){$\cdots$}
  \put(25,0){\line(0,1){7.5}}
  \put(1,7.5){\line(1,0){24}}
\end{picture}
}
    \tilde{K}_\pi(L)
    +\sum_{k=1}^r
    \sum_{\substack{\pi\in \NC^{B_k}(2r)\\ \pi\vee\onetwo{r}=\hat{1}_{2r}}}
    \tilde{K}_\pi(L)
    .
  \end{equation}
 
\noindent
   Every $\pi$ in this sum splits $[2r]$
  into the two intervals $B_k=(1,\dots,2k-1)$ and it is a complement $(2k,\dots,2r)$, hence we have 
  \begin{align*}
    \sum_{\substack{\pi\in \NC^{B_k}(2r)\\ \pi\vee\onetwo{r}=\hat{1}_{2r}}} \tilde{K}_\pi(L)
    &=\sum_{\sigma_1,\sigma_2,\dots,\sigma_k\in\SG_n} {K}_{2k-1}(L_{\sigma_1},L_{\sigma_1},\dots,L_{\sigma_{k-1}},L_{\sigma_{k-1}},L_{\sigma_k})\\& \times 
\sum_{\sigma_{k+1},\dots,\sigma_r\in\SG_n}    
    \sum_{\substack{\pi\in \NC(2(r-k)+1)\\ \pi\vee \begin{picture}(62,3.5)(1,0)
  \put(2,0){\line(0,1){4.5}}
  \put(8,0){\line(0,1){4.5}}
  \put(14,0){\line(0,1){4.5}}
  \put(20,0){\line(0,1){4.5}}
  \put(26,0){\line(0,1){4.5}}
\put(33.5,0){$\cdots$}
  \put(56,0){\line(0,1){4.5}}
  \put(62,0){\line(0,1){4.5}}
  \put(2,4.5){\line(1,0){0}}
  \put(8,4.5){\line(1,0){6}}
  \put(20,4.5){\line(1,0){6}}
  \put(56,4.5){\line(1,0){6}}
\end{picture}=\hat{1}_{2(r-k)+1}}}
               K_\pi(L_{\sigma_k},L_{\sigma_{k+1}},L_{\sigma_{k+1}},
                    \dots,
                    L_{\sigma_r},L_{\sigma_r})
\\&    
   =\sum_{\sigma_1,\sigma_2,\dots,\sigma_k\in\SG_n} {K}_{2k-1}(L_{\sigma_1},L_{\sigma_1},\dots,L_{\sigma_{k-1}},L_{\sigma_{k-1}},L_{\sigma_k})
   \times K_{r-k+1}(L_{\sigma_k},P,\dots,P).
  \end{align*}
Finally,   by multilinearity,     the factor
$$
K_{r-k+1}(L_{\sigma_k},P,P,\dots,P)=\sum_{i=1}^n \alpha_i K_{r-k+1}(X_{\sigma_k (i)},P,\dots,P) =K_{r-k+1}(X_1,P,\dots,P)\sum_{i=1}^n \alpha_i 
$$
vanishes for every $\sigma_k$ because  in the last equality we use the fact that  $P$ is symmetric polynomial in variables $X_1,\dots, X_n.$
\end{proof}

\begin{cor}\label{wniosekkumulanty}
  Let $P$ be as in Proposition \ref{lem:oddcancellation}, then 
  
\begin{align*}
K_r(P)&= K_r\big((n-1)!(\alpha_1^2+\dots+\alpha_n^2)(X_1^2+\dots+X_n^2)+(n-2)!(\sum_{i\neq j}\alpha_i \alpha_j)\sum_{i\neq j}X_iX_j\big)
\intertext{by Proposition \ref{lem:oddcancellation}}
&=
K_r\big((n-1)!(\alpha_1^2+\dots+\alpha_n^2)(X_1^2+\dots+X_n^2)\big)
\intertext{and by multilinearity  and Boolean independence }&=
n(n-1)!^r(\alpha_1^2+\dots+\alpha_n^2)^r
K_{2r}(X).
\end{align*}
In the case of   the sample variance $Q_n
=\sum_{i=1}^n(\X_i-\overline{\X})^2$
we apply Proposition~\ref{lem:oddcancellation} for $L=X_1-\overline{X}$ and 
$P=\sum_{\sigma\in\SG_n} L_\sigma^2=(n-1)!\,Q_n$ 
to conclude that
\begin{align*}
K_r(P)&=
n(n-1)!^r\Big(\Big(1-\frac{1}{n}\Big)^2 +\underbrace{\frac{1}{n^2}+\dots+\frac{1}{n^2}}_{n-1}\Big)^rK_{2r}(X)
\\&=n(n-1)!^r\Big(1-\frac{1}{n}\Big)^rK_{2r}(X)
\end{align*}
and from above we have 
$$
K_r(Q_n) =n
\left(
1-\frac{1}{n}
\right)^r 
K_{2r}(X).
$$
\end{cor}
\noindent
The main result of \cite{EjsmontLehner:2019:commutators} is a characterization of quadratic forms, which exhibits the phenomenon of cancellation of odd cumulants, i.e.,
whose distributions do not depend on the odd cumulants of the distributions of the  arbitrary free random variables.   
Note that the similar result for 
free identically distributed families is still open. 
Now, we present the solution of this problem  in a Boolean version, which indicates the  direction to go in free probability.   

\begin{theo} 
  \label{lem:equvalencecancellation}
Let $\X_1, \X_2,\dots, \X_n$ be Boolean independent identically distributed copies
  of a random variable $X$,
   $A=[a_{i,j}]_{i,j=1}^n\in M_n^{sa}(\C)$ 
  and $T_n=\sum a_{i,j}X_iX_j$ a quadratic form.  
    Then  the $r^{th}$ cumulants of the $T_n$ 
have only one cotributing  
term  $\hat{1}_{2r}=
\begin{picture}(25,6.5)(1,0)
  \put(1,0){\line(0,1){7.5}}
  \put(6,0){\line(0,1){7.5}}
  \put(8,0){$\cdots$}
  \put(25,0){\line(0,1){7.5}}
  \put(1,7.5){\line(1,0){24}}
\end{picture}
$, if and only if  
 $A$ is a zero sum matrix with constant diagonal.   

\end{theo} 
\begin{proof}
  Let  $\hat\pi$ be the image of $\pi\in\NC(r+1)$ under the bijection introduced in Lemma~\ref{lemm:singelton}, 
  then $\hat \pi$ has two singletons if and only if $\pi$ consists only singletons, namely
   \begin{align*}
    \hat \pi =
\begin{picture}(56,6.5)(1,0)
  \put(2,0){\line(0,1){4.5}}
  \put(8,0){\line(0,1){4.5}}
  \put(14,0){\line(0,1){4.5}}
  \put(20,0){\line(0,1){4.5}}
  \put(26,0){\line(0,1){4.5}}
  \put(44,0){\line(0,1){4.5}}
  \put(50,0){\line(0,1){4.5}}
  \put(56,0){\line(0,1){4.5}}
\put(28.0,0){$\cdots$}
  \put(8,4.5){\line(1,0){6}}
  \put(20,4.5){\line(1,0){6}}
  \put(44,4.5){\line(1,0){6}}
\end{picture}\in\NC(2r) \longleftrightarrow \pi=
\begin{picture}(26,3.5)(1,0)
  \put(2,0){\line(0,1){4.5}}
  \put(8,0){\line(0,1){4.5}}
\put(10,0){$\cdots$}
  \put(26,0){\line(0,1){4.5}}
  \put(2,4.5){\line(1,0){0}}
  \put(8,4.5){\line(1,0){0}}
  \put(14,4.5){\line(1,0){0}}
  \put(20,4.5){\line(1,0){0}}
  \put(26,4.5){\line(1,0){0}}
\end{picture} \in\NC(r+1)  
  \end{align*}  and so by Proposition \ref{prop:CykliczneVariancja} 
 we can write 
  \begin{align}\label{eq:dowodrownowaznsoci}
    K_r(T_n) = \Tr(J_n A_n^r)K_1^2(X) K_2^{r-1}(X)+\sum_{ \substack{\pi\in \NC(r+1)\\ \pi> 
\begin{picture}(26,3.5)(1,0)
  \put(2,0){\line(0,1){4.5}}
  \put(8,0){\line(0,1){4.5}}
\put(10,0){$\cdots$}
  \put(26,0){\line(0,1){4.5}}
  \put(2,4.5){\line(1,0){0}}
  \put(8,4.5){\line(1,0){0}}
  \put(14,4.5){\line(1,0){0}}
  \put(20,4.5){\line(1,0){0}}
  \put(26,4.5){\line(1,0){0}}
\end{picture} }}\Tr (\EB[\overline{\Krewl \pi} ] (\underbrace{\J,A,\dots,A}_{r+1}))K_{\hat\pi}(X).
  \end{align} 
If we assume  that the cumulants of $T_n$ 
have only one contributing term, then the factor related to $K_1^2(X)$ must vanish, which implies that  $\Tr(\J A^{r})=0,$ for all $r\in \N$ and by Lemma \ref{lemma:equvalenttwoforall} $A$ is zero sum matrix. We futher  extract from the formula \eqref{eq:dowodrownowaznsoci},  this partition $\sigma$ from $\NC(r+1)$ which has two odd blocks of the same  size, namely    \begin{align*}
    \hat \sigma =\begin{picture}(32,6.5)(1,0)
  \put(0,0){\line(0,1){4.5}}
  \put(6,0){\line(0,1){4.5}}
  \put(12,0){\line(0,1){4.5}}
  \put(18,0){\line(0,1){4.5}}
  \put(36,0){\line(0,1){4.5}}
  \put(42,0){\line(0,1){4.5}}
\put(20,0){$\cdots$}
  \put(0,4.5){\line(1,0){6}}
  \put(12,4.5){\line(1,0){6}}
  \put(36,4.5){\line(1,0){6}}
  \put(44,0){$\cdots$}
 \put(60,0){\line(0,1){4.5}}
  \put(66,0){\line(0,1){4.5}}
   \put(60,4.5){\line(1,0){6}}
  \put(72,0){\line(0,1){4.5}}
  \put(78,0){\line(0,1){4.5}}
   \put(72,4.5){\line(1,0){6}}
 
   \put(0,8){\line(0,1){4.5}}
  \put(6,8){\line(0,1){4.5}}
  \put(12,8){\line(0,1){4.5}}
  \put(18,8){\line(0,1){4.5}}
  \put(36,8){\line(0,1){4.5}}
  \put(0,12.5){\line(1,0){36}}
  
 
   \put(42,8){\line(0,1){4.5}}
  \put(60,8){\line(0,1){4.5}}
  \put(66,8){\line(0,1){4.5}}
  \put(72,8){\line(0,1){4.5}}
  \put(78,8){\line(0,1){4.5}}
  \put(42,12.5){\line(1,0){36}}
  
\end{picture} \qquad \qquad \in\NC(2r) \longleftrightarrow \sigma=
\begin{picture}(26,3.5)(1,0)
  \put(2,0){\line(0,1){4.5}}
  \put(8,0){\line(0,1){4.5}}
\put(10,0){$\cdots$}
  \put(26,0){\line(0,1){4.5}}

   \put(32,0){\line(0,1){4.5}}
\put(34,0){$\cdots$}
 \put(50,0){\line(0,1){4.5}}
  \put(56,0){\line(0,1){4.5}}
  
  \put(2,8){\line(0,1){4.5}}
  \put(8,8){\line(0,1){4.5}}

  \put(26,8){\line(0,1){4.5}}

   \put(32,8){\line(0,1){4.5}}

 \put(50,8){\line(0,1){4.5}}
  \put(56,8){\line(0,1){4.5}}
 \put(2,12.5){\line(1,0){24}} 
 \put(32,12.5){\line(1,0){24}} 
\end{picture} \quad \quad \quad  \in\NC(r+1)\quad  \text{ for }r \text{ odd}.
  \end{align*}  
  Using this, we further expand \eqref{eq:dowodrownowaznsoci}, for odd  $r$  and obtain
    \begin{align*}
    K_r(T_n) &= \Tr (\EB[\overline{\Krewl \sigma} ] (\underbrace{\J,A,\dots,A}_{r+1}))K_{\hat\sigma}(X)+ \sum_{ \substack{\pi\in \NC(r+1)\\ \pi> 
\begin{picture}(26,3.5)(1,0)
  \put(2,0){\line(0,1){4.5}}
  \put(8,0){\line(0,1){4.5}}
\put(10,0){$\cdots$}
  \put(26,0){\line(0,1){4.5}}
  \put(2,4.5){\line(1,0){0}}
  \put(8,4.5){\line(1,0){0}}
  \put(14,4.5){\line(1,0){0}}
  \put(20,4.5){\line(1,0){0}}
  \put(26,4.5){\line(1,0){0}}
\end{picture}  \\\pi \neq \sigma }}\Tr (\EB[\overline{\Krewl \pi} ] (\underbrace{\J,A,\dots,A}_{r+1}))K_{\hat\pi}(X)    
 \intertext{and observe that $\overline{\Krewl \sigma}=\begin{picture}(32,6.5)(1,0)
   \put(0,0){\line(0,1){8}}
  \put(6,0){\line(0,1){4.5}}
  \put(12,0){\line(0,1){4.5}}
 \put(16,-2){$\cdots$}
  \put(32,0){\line(0,1){4.5}}
  \put(38,0){\line(0,1){4.5}}
     \put(44,0){\line(0,1){8}}
  \put(6,4.5){\line(1,0){32}}
   \put(0,8){\line(1,0){44}}

  \put(50,0){\line(0,1){4.5}}
  \put(56,0){\line(0,1){4.5}}
 \put(60,-2){$\cdots$}
  \put(76,0){\line(0,1){4.5}}
  \put(82,0){\line(0,1){4.5}}
    \put(50,4.5){\line(1,0){32}}
\end{picture}\qquad \qquad \text{ ,}$ which yields } &= \Tr (\J \EB(\underbrace{A,\dots,A}_{\frac{r-1}{2}})A \EB(\underbrace{A,\dots,A}_{\frac{r-1}{2}}))K_{r}^2(X)
    + \sum_{ \substack{\pi\in \NC(r+1)\\ \pi> 
\begin{picture}(26,3.5)(1,0)
  \put(2,0){\line(0,1){4.5}}
  \put(8,0){\line(0,1){4.5}}
\put(10,0){$\cdots$}
  \put(26,0){\line(0,1){4.5}}
  \put(2,4.5){\line(1,0){0}}
  \put(8,4.5){\line(1,0){0}}
  \put(14,4.5){\line(1,0){0}}
  \put(20,4.5){\line(1,0){0}}
  \put(26,4.5){\line(1,0){0}}
\end{picture}  \\\pi \neq \sigma }}\Tr (\EB[\overline{K\Krewl \pi} ] (\underbrace{\J,A,\dots,A}_{r+1}))K_{\hat\pi}(X) 
    \\ &= \Tr (\J \ED(A^{\odot\frac{r-1}{2}})A \ED(A^{\odot\frac{r-1}{2}}))K_{r}^2(X)
    + \sum_{ \substack{\pi\in \NC(r+1)\\ \pi> 
\begin{picture}(26,3.5)(1,0)
  \put(2,0){\line(0,1){4.5}}
  \put(8,0){\line(0,1){4.5}}
\put(10,0){$\cdots$}
  \put(26,0){\line(0,1){4.5}}
  \put(2,4.5){\line(1,0){0}}
  \put(8,4.5){\line(1,0){0}}
  \put(14,4.5){\line(1,0){0}}
  \put(20,4.5){\line(1,0){0}}
  \put(26,4.5){\line(1,0){0}}
\end{picture}  \\\pi \neq \sigma }}\Tr (\EB[\overline{\Krewl \pi} ] (\underbrace{\J,A,\dots,A}_{r+1}))K_{\hat\pi}(X). 
  \end{align*} 
Since we consider the interval partitions,  there is just one coefficient appearing in the term with  $K_{r}^2(X)$.  Thus we see that $$\Tr (\J \ED(A^{\odot\frac{r-1}{2}})A \ED(A^{\odot\frac{r-1}{2}}))=\Tr (\J \diag(A)^{\frac{r-1}{2}} A \diag(A)^{\frac{r-1}{2}})=0$$ for all odd $r$.
Thus from Lemma 
\ref{lemm:macierzediagonalne}, this implies that  $A$ is constant diagonal matrix.   

\noindent
The implication in the other direction is a simple manipulation with a constant diagonal matrix.
The  Hadamard cumulants which are not associated with the first block    are of the form 
$$\ED(A^{\odot k})=\left[\begin{smallmatrix}
    a^k & 0\\
    0& a^k
    \end{smallmatrix}\right]_n.$$  
From this we see that the  Hadamard cumulants  are commutative and  formula \eqref{eq:kumulantsamplevariance2} can be rewritten   as
\begin{align}\label{eq:PomDowoddiagonal} 
K_r(T_n) =\sum_{ \pi\in \NC(r+1)} \Tr(\J A^{|B_{1}|-1})  a^{r+1-|B_{1}|}
K_{\hat\pi}(X)
    \end{align}
    where $B_1$ is the first block of $\overline K(\pi)$. By Lemma \ref{lemma:equvalenttwoforall} we have  $\Tr(\J A^{|B_{1}|-1})=0$ whenever size of block $B_1$ is bigger  then one, and from this we conclude that the   formula \eqref{eq:PomDowoddiagonal}  takes non zero value when 
    $B_1=(1)$, i.e. $\pi=
\begin{picture}(25,6.5)(1,0)
  \put(1,0){\line(0,1){7.5}}
  \put(6,0){\line(0,1){7.5}}
  \put(8,0){$\cdots$}
  \put(25,0){\line(0,1){7.5}}
  \put(1,7.5){\line(1,0){24}}
\end{picture}
$, and in this case  we have  $K_r(T_n) =na^r  K_{2r}(X)$.
\end{proof}
\begin{cor}These findings provide
one more argument for the cancellation phenomenon  of the sample variance 
$$
\Q_n=\sum_{i=1}^n(\X_i-\overline{\X})^2
=\Big(1-\frac{1}{n}\Big)\sum_{i=1}^n\X_i^2-\frac{1}{n}\sum_{i,j=1, \textrm{ }i\neq j}^n\X_i\X_j.
$$
The corresponding matrix from Theorem \ref{lem:equvalencecancellation} is $A=\left[\begin{smallmatrix}
    1-\frac{1}{n} & -\frac{1}{n}\\
    -\frac{1}{n}& 1-\frac{1}{n} \\
    \end{smallmatrix}\right]_n$. It is easily verified that $A$ is  an orthogonal projection of rank $n-1$ and we see  that $\Tr(\J A)=\Tr(\J A^{2})=0.$
\end{cor}

\begin{Rem}
 We note that a cancellation  phenomenon occurs also for the free commutator  \cite{NicaSpeicher:1998,EjsmontLehner:2019:commutators} but this  is not true in the case
of  Boolean probability. Indeed, if $X$ and $Y$ are Boolean independent, then $$K_2(i(XY-YX))=K_1^2(X)K_2(Y)+K_1^2(Y)K_2(X).$$ But from Theorem
\ref{lem:equvalencecancellation}, we conclude that for identical distributed random variables $X_1,X_2,X_3$ cancellation phenomenon is true for 
$i(X_1X_2-X_2X_1)+i(X_3X_1-X_1X_3)+i(X_2X_3-X_3X_2)$  because the corresponding matrix is  
$$\left[\begin{smallmatrix}
    0 & i & -i \\
    -i& 0 & i  \\
    i & -i & 0
    \end{smallmatrix}\right].$$
\end{Rem}
\noindent
We conclude this section with the Boolean  analogue of the $\chi^2$-conjecture ()see \cite{EjsmontLehner:2017}). 
\begin{prop} 
 Let $\X_1, \X_2,\dots, \X_n \in \A_{sa}$ be the Boolean  copies of
a random variable $X$ with  variance $1$.
 Then $\Q_n=\sum_{i=1}^n(\X_i-\overline{\X})^2$
is distributed according to  Boolean Poisson, with  rate $n$ and the jump size $1-\frac{1}{n}$, if and only if $\X$ is even  Poisson random variable. 
\label{twr:OddElement} 
\end{prop}
\begin{proof}
Let  $\X_1,\dots,\X_n$ be Boolean copies of a fixed random variable $X$.
In this case Corollary~\ref{wniosekkumulanty} and the assumption of Theorem \ref{twr:OddElement}  imply that
$$
K_r(Q_n) = n
\left(
1-\frac{1}{n}
\right)^r 
K_{2r}(X)= n
\left(
1-\frac{1}{n}
\right)^r.
$$
From this we infer that
$K_{2r}(X)=1$.
Conversely, 
suppose that $\X_i$'s are even  Poisson, then from  Corollary~\ref{wniosekkumulanty}
 we get 
$ K_r(\Q_n) 
  = n
\left(
1-\frac{1}{n}
\right)^r
$.

\end{proof}

\section{Limit theorems  for quadratic forms}

\subsection{A general  Limit Theorem}



In this section we consider  limit theorems for the sums of 
quadratic forms of the following type.
\begin{theo}
  \label{thm:quadraticCLT}
  Let $A_n = [a_{i,j}^{(n)}]\in M_n^{sa}(\C)$ with zero diagonal,  uniformly bounded entries i.e. 
  $\sup_{i,j,n} \bigabs{a_{i,j}^{(n)}}<\infty$ and such that the matrix
  $\frac{1}{n}A_n$ has limit distribution $\mu$
  with respect to the state $\omega$. 
  Let $X_1,\dots, X_n$ be  identically distributed Boolean independent random variables,  with mean $\state(X_i)=\frac{1}{\sqrt{n}}$  and variance equal 1.  Additionally, let Boolean cumulants satisfy the product formula of Theorem~\ref{thm:krawczyk}. Then
  the sequence of quadratic forms
\begin{align*}
  Q_n  &= \frac{1}{n}\sum_{\substack{i,j=1\\ i\neq j} }^n a_{i,j}^{(n)} X_iX_j
\intertext{converges in distribution to $Y$, where}  
  K_r(Y) &=  \int
  t^r \dd{\mu(t)}.
\end{align*}
\end{theo}
\begin{Rem}
Note that assumption  $a_{i,i}=0$  plays a crucial
role in the proof of Theorem \ref{thm:quadraticCLT}, which is in opposition to \cite[Theorem 3.1]{EjsmontLehner:2020:tangent}. We may add diagonal elements but then we need to assume the existence of a limit  of the first moment, such as  
$$\lim_{n\to \infty}K_1(Q_n)=\lim_{n\to \infty}\frac{1}{n}\left(\sum_{i, j}a_{i,j}/n+\sum_{i}a_{i,i}\right). 
    $$
We do not do this because our main example does not have diagonal elements. 
\end{Rem}

\begin{proof}
 The random variables $X_1,\dots, X_n$ have the same distribution (which depends on $n$), so we can use the product formula from Proposition \ref{prop:CykliczneVariancja}. Then 
  \begin{align*}
    K_r(Q_n)
            &=\frac{1}{n^{r}}
                \sum_{ \pi\in \NC(r+1)}
      \left(\sum_{\ker\underline{i}\geq \pi}\Tr(\J E_{i_0} AE_{i_1}\dots AE_{i_r})\right) K_{\hat\pi}(X).
  \end{align*}
 where   $\hat\pi$ is the image of $\pi\in\NC(r+1)$ under the bijection introduced in Lemma~\ref{lemm:singelton}.
 This in turn implies that there
  are only $n^{\#{\pi}}$ allowed choices of indices $\underline{i}$
  and we have the following estimate
  \begin{align} \label{eq:estymacja}
\begin{split}    
    \abs{\frac{1}{n^{r}}
       \left(\sum_{\ker\underline{i}\geq \pi}\Tr(\J E_{i_0} AE_{i_1}\dots AE_{i_r})\right)
         K_{\hat \pi}(X)}
    &\leq 
      n^{\#{\pi}-r} C^r
      \abs{K_{\hat \pi}(X)}
\\    
      &=
      n^{\#{\pi}-r-\frac{1}{2}\#\Sing(\hat \pi)} C^r
      \abs{K_{\hat \pi\setminus \Sing(\hat \pi) } (X)}
      \end{split}
  \end{align}
  where $C=\sup_{i,j,n} \bigabs{a_{i,j}^{(n)}}$ and in the last line we use the assumption $\state(X_i)=\frac{1}{\sqrt{n}}$.  Notation $\hat \pi\setminus \Sing(\hat \pi)$ means the difference of two sets $\hat \pi$ and  $\Sing(\hat \pi)$.
 By Lemma \ref{lemm:singelton}, we have
  $$
\#{\pi} = \left\{ \begin{array}{ll}
1 & \textrm{if $\hat \pi$ does not  have odd blocks, namely } \hat \pi=
\begin{picture}(25,6.5)(1,0)
  \put(1,0){\line(0,1){7.5}}
  \put(6,0){\line(0,1){7.5}}
  \put(8,0){$\cdots$}
  \put(25,0){\line(0,1){7.5}}
  \put(1,7.5){\line(1,0){24}}
\end{picture}
,\\
\leq r-1 & \textrm{if $\hat \pi$ does not  have  singletons,}\\
\leq r & \textrm{if $\hat \pi$  has at most one singleton,}\\
\leq r+1 & \textrm{if $\hat \pi$ has two singletons.}
\end{array} \right.
$$  
\noindent
Figure \ref{fig:partitionspi} presents the examples of partitions of $\pi$ and $\hat{\pi}$.
 
 \begin{figure}[h]
\begin{subfigure}[t]{0.3\textwidth}
 
\hspace{7em}  (a)

\vspace{1em} 

 \centering 
 
\begin{picture}(44,3.5)(1,0)
  \put(2,0){\line(0,1){4.5}}
  \put(8,0){\line(0,1){4.5}}
  \put(14,0){\line(0,1){4.5}}
  \put(20,0){\line(0,1){4.5}}
\put(21.5,0){$\cdots$}
  \put(38,0){\line(0,1){4.5}}
  \put(44,0){\line(0,1){4.5}}
  \put(50,0){\line(0,1){4.5}}
  \put(56,0){\line(0,1){4.5}}
    \put(62,0){\line(0,1){4.5}}
  \put(68,0){\line(0,1){4.5}}
  \put(70,0){$\cdots$}
    \put(86,0){\line(0,1){4.5}}
  \put(92,0){\line(0,1){4.5}}
    \put(98,0){\line(0,1){4.5}}
  \put(104,0){\line(0,1){4.5}}
  
  \put(2,4.5){\line(1,0){6}}
  \put(14,4.5){\line(1,0){6}}
  \put(38,4.5){\line(1,0){6}}
 \put(50,4.5){\line(1,0){6}}
  \put(62,4.5){\line(1,0){6}}
  \put(86,4.5){\line(1,0){6}}
  \put(98,4.5){\line(1,0){6}}

    \put(2,8){\line(0,1){4.5}}
    \put(8,8){\line(0,1){4.5}}
  \put(14,8){\line(0,1){4.5}}
  \put(20,8){\line(0,1){4.5}}
  \put(26,8){\line(0,1){4.5}}
  \put(28,6){$\cdots$}
  
  \put(44,8){\line(0,1){4.5}}
  \put(50,8){\line(0,1){4.5}}
  \put(56,8){\line(0,1){4.5}}
  \put(62,8){\line(0,1){4.5}}
  \put(64,6){$\cdots$}
  \put(80,8){\line(0,1){4.5}}
  \put(86,8){\line(0,1){4.5}}
  \put(92,8){\line(0,1){4.5}}
  \put(98,8){\line(0,1){4.5}}
  \put(104,8){\line(0,1){4.5}}

  \put(2,12.5){\line(1,0){6}}
  \put(8,12.5){\line(1,0){6}}
  \put(20,12.5){\line(1,0){6}}
  \put(44,12.5){\line(1,0){6}}
  \put(56,12.5){\line(1,0){6}}
   \put(80,12.5){\line(1,0){6}}
  \put(92,12.5){\line(1,0){6}}
    \put(98,12.5){\line(1,0){6}}

\end{picture}

 \end{subfigure}
~
  \begin{subfigure}[t]{0.3\textwidth}
\hspace{7em}  (b)
 
\vspace{1em}  
 \centering
\begin{picture}(44,3.5)(1,0)
  \put(2,0){\line(0,1){4.5}}
  \put(8,0){\line(0,1){4.5}}
  \put(14,0){\line(0,1){4.5}}
  \put(20,0){\line(0,1){4.5}}
\put(21.5,0){$\cdots$}
  \put(38,0){\line(0,1){4.5}}
  \put(44,0){\line(0,1){4.5}}
  \put(50,0){\line(0,1){4.5}}
  \put(56,0){\line(0,1){4.5}}
    \put(62,0){\line(0,1){4.5}}
  \put(68,0){\line(0,1){4.5}}
  \put(70,0){$\cdots$}
    \put(86,0){\line(0,1){4.5}}
  \put(92,0){\line(0,1){4.5}}
    \put(98,0){\line(0,1){4.5}}
  \put(104,0){\line(0,1){4.5}}
  
  \put(2,4.5){\line(1,0){6}}
  \put(14,4.5){\line(1,0){6}}
  \put(38,4.5){\line(1,0){6}}
 \put(50,4.5){\line(1,0){6}}
  \put(62,4.5){\line(1,0){6}}
  \put(86,4.5){\line(1,0){6}}
  \put(98,4.5){\line(1,0){6}}

    \put(2,8){\line(0,1){4.5}}
    \put(8,8){\line(0,1){4.5}}
  \put(14,8){\line(0,1){4.5}}
  \put(20,8){\line(0,1){4.5}}
  \put(26,8){\line(0,1){4.5}}
  \put(28,6){$\cdots$}
  
  \put(44,8){\line(0,1){4.5}}
  \put(50,8){\line(0,1){4.5}}
  \put(56,8){\line(0,1){4.5}}
  \put(62,8){\line(0,1){4.5}}
  \put(64,6){$\cdots$}
  \put(80,8){\line(0,1){4.5}}
  \put(86,8){\line(0,1){4.5}}
  \put(92,8){\line(0,1){4.5}}
  \put(98,8){\line(0,1){4.5}}
  \put(104,8){\line(0,1){4.5}}

  \put(8,12.5){\line(1,0){6}}
  \put(20,12.5){\line(1,0){6}}
  \put(44,12.5){\line(1,0){6}}
  \put(56,12.5){\line(1,0){6}}
   \put(80,12.5){\line(1,0){6}}
  \put(92,12.5){\line(1,0){6}}
    \put(98,12.5){\line(1,0){6}}

\end{picture}

 \end{subfigure}
 ~
  \begin{subfigure}[t]{0.3\textwidth}
\hspace{7em}  (c)
 
\vspace{1em}  
 \centering
\begin{picture}(44,3.5)(1,0)
  \put(2,0){\line(0,1){4.5}}
  \put(8,0){\line(0,1){4.5}}
  \put(14,0){\line(0,1){4.5}}
  \put(20,0){\line(0,1){4.5}}
\put(21.5,0){$\cdots$}
  \put(38,0){\line(0,1){4.5}}
  \put(44,0){\line(0,1){4.5}}
  \put(50,0){\line(0,1){4.5}}
  \put(56,0){\line(0,1){4.5}}
    \put(62,0){\line(0,1){4.5}}
  \put(68,0){\line(0,1){4.5}}
  \put(70,0){$\cdots$}
    \put(86,0){\line(0,1){4.5}}
  \put(92,0){\line(0,1){4.5}}
    \put(98,0){\line(0,1){4.5}}
  \put(104,0){\line(0,1){4.5}}
  
  \put(2,4.5){\line(1,0){6}}
  \put(14,4.5){\line(1,0){6}}
  \put(38,4.5){\line(1,0){6}}
 \put(50,4.5){\line(1,0){6}}
  \put(62,4.5){\line(1,0){6}}
  \put(86,4.5){\line(1,0){6}}
  \put(98,4.5){\line(1,0){6}}

    \put(2,8){\line(0,1){4.5}}
    \put(8,8){\line(0,1){4.5}}
  \put(14,8){\line(0,1){4.5}}
  \put(20,8){\line(0,1){4.5}}
  \put(26,8){\line(0,1){4.5}}
  \put(28,6){$\cdots$}
  
  \put(44,8){\line(0,1){4.5}}
  \put(50,8){\line(0,1){4.5}}
  \put(56,8){\line(0,1){4.5}}
  \put(62,8){\line(0,1){4.5}}
  \put(64,6){$\cdots$}
  \put(80,8){\line(0,1){4.5}}
  \put(86,8){\line(0,1){4.5}}
  \put(92,8){\line(0,1){4.5}}
  \put(98,8){\line(0,1){4.5}}
  \put(104,8){\line(0,1){4.5}}

  \put(8,12.5){\line(1,0){6}}
  \put(20,12.5){\line(1,0){6}}
  \put(44,12.5){\line(1,0){6}}
  \put(56,12.5){\line(1,0){6}}
   \put(80,12.5){\line(1,0){6}}
  \put(92,12.5){\line(1,0){6}}

\end{picture}

 \end{subfigure}
\caption{Examples of partitions of $\pi$ and $\hat{\pi}$: a) $\hat{\pi}$ doesn't have singletons, \\b) $\hat{\pi}$ has at most one singleton and c) $\hat{\pi}$ has two singletons.}
\label{fig:partitionspi}
\end{figure}
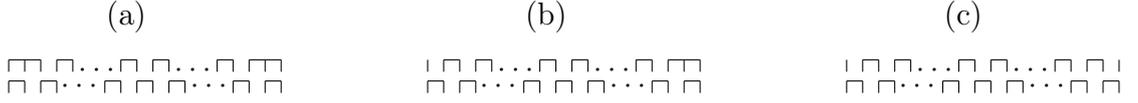
\noindent
 We see that  the  factor $n^{\#{\pi}-r-\frac{1}{2}\#\Sing(\hat \pi)}$  converges to zero whenever $\hat\pi\neq 
\begin{picture}(56,6.5)(1,0)
  \put(2,0){\line(0,1){4.5}}
  \put(8,0){\line(0,1){4.5}}
  \put(14,0){\line(0,1){4.5}}
  \put(20,0){\line(0,1){4.5}}
  \put(26,0){\line(0,1){4.5}}
  \put(44,0){\line(0,1){4.5}}
  \put(50,0){\line(0,1){4.5}}
  \put(56,0){\line(0,1){4.5}}
\put(28.0,0){$\cdots$}
  \put(8,4.5){\line(1,0){6}}
  \put(20,4.5){\line(1,0){6}}
  \put(44,4.5){\line(1,0){6}}
\end{picture} $   (as well as expression \eqref{eq:estymacja} tends to zero) as $n\to\infty$.
 The cases   $r=1$ and $\# \Sing(\hat \pi)=0$ should be considered separately, because  then  we have one more contributing term $\pi=(1,2)$, but since we assume that the matrices $A_n$  have  zero diagonal, we may skip this case. 
  On the other hand, $\#{\pi}=r+1$ if and only if $\pi$  is the singleton partition, equivalently $\hat\pi= 
\begin{picture}(56,6.5)(1,0)
  \put(2,0){\line(0,1){4.5}}
  \put(8,0){\line(0,1){4.5}}
  \put(14,0){\line(0,1){4.5}}
  \put(20,0){\line(0,1){4.5}}
  \put(26,0){\line(0,1){4.5}}
  \put(44,0){\line(0,1){4.5}}
  \put(50,0){\line(0,1){4.5}}
  \put(56,0){\line(0,1){4.5}}
\put(28.0,0){$\cdots$}
  \put(8,4.5){\line(1,0){6}}
  \put(20,4.5){\line(1,0){6}}
  \put(44,4.5){\line(1,0){6}}
\end{picture}$ and finally, by Corollary  \ref{eq:semisirgular}, we have
\begin{align*}
  K_r(Q_n)  
    &= \frac{1}{n^{r+1}}\Tr(J_n A_n^r) + \mathcal{O}(1/n) 
  \\&= \omega( \frac{1}{n^{r}}A_n^r) + \mathcal{O}(1/n) \xrightarrow[n\to\infty]{} \int
  t^r \dd{\mu(t)}.
\end{align*}
\end{proof}

\subsection{Generating functions}
\label{sec:genfunc}
In this subsection we compute  the  moment generating function of the matrix $a\P+b B_n$
with respect to the state
$\omega$, that is
\begin{align*}
F_{a\P+ b B_n}(z) &= \omega((I-z(a\P+b B_n))^{-1})= \Tr(\P (I-z(a \P+b B_n))^{-1}),\quad a,b\in \R .
\end{align*}
\noindent
In \cite[Lemma 4.3 and Lemma 4.4]{EjsmontLehner:2019:cotangent}, 
the moment generating function $F_{B_n}(z)$ was computed using the cyclic Boolean convolution 
and it is equal to 
\begin{align}\label{eq:Fcotanges}
 F_{B_n}(z)=\frac{\tan(n\arctan \frac{z}{n})}{z}.
\end{align}
Now we compute it  in a more general case.
\begin{lemm}\label{lemm:funkcjagrnerujaca}
  The moment generating function of $a\P+ b B_n$ is equal to
  \begin{align}
  \begin{split}
F_{a\P+ b B_n}(z) =\frac{{\tan(n\arctan \frac{bz}{n})}}{bz-az{\tan(n\arctan \frac{bz}{n})}}.
\end{split}  
  \end{align}
\end{lemm}
\begin{proof}
  Note that $P_n$ is a self-adjoint projection
  of rank $1$. 
  It follows that any mixed moment can be expressed as (see also  \cite[proof of Theorem 3.1]{EjsmontLehner:2019:cotangent})
  \begin{align*}
   \Tr(P_n^{k_1} B_n^{l_1}P_n^{k_2} B_n^{l_2}\dotsm P_n^{k_r}B_n^{l_r})    
    &= \Tr(\P B_n^{l_1}\P B_n^{l_2}\dotsm \P B_n^{l_r})  \\       
    &=\Tr(P_n B_n^{l_1})\Tr(P_n B_n^{l_2})\dotsm \Tr( P_n B_n^{l_r}) . 
  \end{align*}
  The first terms of the power series are easy to calculate, $\omega(a\P+ b B_n)=a\omega(\P)+b\omega(B_n)=a$, and we have 
  \begin{equation}
    \label{eq:MxJnBn.1}
   F_{a\P+ b B_n}(z)=  1 + a  z + \sum_{m\geq 2}     \Tr(\P( a\P+ b  B_n)^m) z^m.
  \end{equation}
 For $m\geq2$ we expand the powers and arrange the resulting words
  according to the last letter:
      \begin{align*}
    \Tr(\P (&a \P+b  B_n)^m) 
    = \Tr\Bigl ( a^m \P +\P( b B_n)^m   \\
    & \phantom{==} + \sum_{\substack{
           k\geq1\\
           p_0\geq 0\\
           p_1,p_2,\dots,p_k\geq 1\\
           q_1,q_2,\dots,q_k\geq 1\\
           p_0+q_1+p_1+\dots+q_k+p_k=m
         }
       }
       \P (b B_n)^{p_0}(a\P)^{q_1}(b B_n)^{p_1}(a\P)^{q_2}(b B_n)^{p_2}\dotsm (a\P)^{q_k}(b B_n)^{p_k}
       \\
    & \phantom{==} + \sum_{\substack{
           k\geq1\\
           q_0\geq 0\\
           p_1,p_2,\dots,p_k\geq 1\\
           q_1,q_2,\dots,q_k\geq 1\\
           q_0+p_1+q_1+\dots+p_k+q_k=m
         }
       }
      \P (a\P)^{q_0} (bB_n)^{p_1}(a\P)^{q_1}(b B_n)^{p_2}(a\P)^{q_2}\dotsm (b B_n)^{p_k} (a\P)^{q_k} 
       \Bigr)
       \\
    &=  \Tr(a^m \P) +\Tr( \P( b B_n)^m) \\
    & \phantom{==} + \sum_{\substack{
           k\geq1\\
           p_0\geq 0\\
           p_1,p_2,\dots,p_k\geq 1\\
           q_1,q_2,\dots,q_k\geq 1\\
           p_0+q_1+p_1+\dots+q_k+p_k=m
         }
       }
       a^{q_1+\dots+q_k}
       \Tr(P_n B_n^{p_0})b^{p_0}\Tr(P_n B_n^{p_1})b^{p_1}\dotsm \Tr(P_n B_n^{p_{k-1}})b^{p_{k-1}}\Tr(P_n B_n^{p_k})b^{p_k}
       \\
    & \phantom{==} + \sum_{\substack{
           k\geq1\\
           q_0\geq 0\\
           p_1,p_2,\dots,p_k\geq 1\\
           q_1,q_2,\dots,q_k\geq 1\\
           q_0+p_1+q_1+\dots+p_k+q_k=m
         }
       }
       a^{q_0+q_1+\dots+q_k}
       \Tr(P_n B_n^{p_1})b^{p_1} \Tr(P_n B_n^{p_2})b^{p_2}\dotsm \Tr(P_n B_n^{p_k})b^{p_k}.
  \end{align*}
\noindent
 Let $\hat{F}_{B_n}(bz)= \frac{az}{1-az}
      (F_{B_n}(bz)-1)$, then  inserting this expansion into \eqref{eq:MxJnBn.1} we obtain
  \begin{align*}
  F_{a\P+ b B_n}(z)
 &= 1 + az + \sum_{m\geq 2} \Tr(\P B_n^m)(bz)^m + \sum_{m\geq2}(az)^m\\
    & \phantom{==}
      + \sum_{k\geq1}
      \left(
        \frac{az}{1-az}\right)^k 
      F_{B_n}(bz)(F_{B_n}(bz)-1)^{k}
      \\
    & \phantom{==} 
      + \frac{1}{1-az}
      \sum_{k\geq1}
      \left(
        \frac{az}{1-az}\right)^{k}
      (F_{B_n}(bz)-1)^{k}
        \\
    &= F_{B_n}(bz) + \frac{az}{1-az}
       + \left(\frac{1}{1-az}+{F}_{B_n}(bz)
      \right)\frac{\hat{F}_{B_n}(bz)}{1-\hat{F}_{B_n}(bz)}
      \\
    &= 
   \left(\frac{1}{1-az}+{F}_{B_n}(bz)
      \right)\frac{1}{1-\hat{F}_{B_n}(bz)}-1
      \\&=\frac{{F}_{B_n}(bz)}{1-az{F}_{B_n}(bz)},
 \intertext{and by inserting the equation  \eqref{eq:Fcotanges} into above expression, we get}
&=\frac{{\tan(n\arctan \frac{bz}{n})}}{bz-az{\tan(n\arctan \frac{bz}{n})}}.
\end{align*}
\end{proof}

\subsection{The  Limit Theorem for commutators and anticommutators}
Finally, we will  illustrate the  limit theorem~\ref{thm:quadraticCLT}
with some interesting computable cases.

\begin{theo}[Boolean generalized tangent law]\label{twr:CLTCommutatorsAniticommutators}
   Let $X_1,\dots, X_n$ be  identically distributed Boolean independent random variables  with mean $\state(X_i)=\frac{1}{\sqrt{n}}$  and variance  1. Then
  the sequence of quadratic forms
  $$Q_n=\frac{1}{n}\sum_{\substack{k,j=1\\ k<j}}^n\big(a(\X_k\X_j+\X_j\X_k)+ib(\X_k\X_j-\X_j\X_k)\big) 
  \xrightarrow{d} Y,\quad a, b\in\IR $$
  where the $H$-transform of the limit distribution has the form 
    \begin{align} \label{eq:uogulnionyrozkaldtangesa}
  \begin{split}
H_{Y}(z)=\frac{1}{z}\frac{\tan(bz)}{b-a\tan(bz)}-1.
\end{split} 
\end{align} 
\end{theo}

\begin{proof}
  The system matrix is $$\frac{1}{n}A_n=
  \frac{1}{n}\left[\begin{smallmatrix}
    0 &a+ib\\
    a-ib& 0
    \end{smallmatrix}\right]_n=\underbrace{a\P+bB_n}_{D_n}-\underbrace{a/n\left[\begin{smallmatrix}
    1 &0\\
    0 &1
    \end{smallmatrix}\right]_n}_{F_n}.$$ 
Asymptotically, the moments  of $\frac{1}{n}A_n$ and $D_n$, are  equal. Indeed,
$$\omega((D_n-F_n)^r)=\omega(D_n^r)+\sum_{i=1}^r \underbrace{{{r}\choose{i}}\left(\frac{-a}{n}\right)^i   \omega(D_n^{r-i})}_{ \xrightarrow{n\to \infty} 0},$$
while the last factor converges to zero because  we can  use a simple estimate    
$$\abs{\omega(D_n^{r-i})}\leq \omega((\abs{a}\P+\abs{b}P_n)^{r-i})=(\abs{a}+\abs{b})^{r-i}.$$
Thus from  Lemma~\ref{lemm:funkcjagrnerujaca} and Theorem \ref{thm:quadraticCLT}
  we have  
  \begin{align*}
  H_{Y}(z) &=\lim_{n\to\infty}  F_{a\P+ b B_n}(z)-1
  \\&=  \lim_{n\to\infty} \frac{\tan(n\arctan \frac{bz}{n})}{zb-az\tan(n\arctan \frac{bz}{n})}-1
\\& =\frac{\tan(bz)}{zb-az\tan(bz)}-1.
\end{align*}
\end{proof}
\begin{Rem}\begin{enumerate}[1.]
\item We call the limit law $\mu_Y$ the Boolean tangent law if 
 $H_{Y}(z) =\frac{\tan z}{z}-1.$ We would like to emphasize that $\frac{\tanh z }{z}$ is a characteristic function of probability measure see \cite[equation (3)]{PitmanYor2003}. 
 \item  Nevanlinna functions are analytic functions having a non-negative imaginary part in the upper half-plane.
We also would like to mention  that the tangent function is  a  fundamental example  of Nevanlinna functions, see
\cite{Arlinskii,Gesztesy,Donoghue:1974:monotone,Steinmetz}. 
\item An important connection between free and Boolean infinite
divisibility was established by  Bercovici and Pata    \cite {BercoviciPata:1999} (see also Belinschi 
 and  Nica \cite{BelinschiNica2008}) in the form of a bijection  from the
class of  probability measure to the class of free   infinitely divisible laws (but also to the class of classical  infinitely divisible).
The easiest
way to define the Boolean  B-P bijection is as follows. Let $\mu$ be a probability measure  having all
moments, and consider its sequence $c_n$ of Boolean cumulants. Then the map $\Lambda$ can be defined
as the mapping that sends $\mu$ to the probability measure on $\R$ with free cumulants $c_n$.
The inverse image of the free tangent law under the Bercovici-Pata bijection has the following
cumulant  function 
$$R(z)=\frac{\tan z-z}{z^2}.$$
This  means that  they are not the same cumulants as in the free version  \cite[Theorem 4.1]{EjsmontLehner:2020:tangent}, because they are shifted $K_r^{\text{Boolean}}=K_{r+1}^{\text{free}}$ (the  Boolean  B-P bijection is not the Boolean tangent law mapped to the free tangent law).
\item In the general case the Boolean cumulants of \eqref{eq:uogulnionyrozkaldtangesa} can be expressed in terms of the generating function of the higher order
tangent numbers as    
$$
  K_r(Y)
  = b^{r}\frac{T_{r+1}(a/b)}{(r+1)!}, \quad r\in \N
  $$ where the polynomials
   ${T_r(x)}$ were defined in Section \ref{ssec:tangentnumbers}. This follows from simple manipulations
using the combinatorics of tangent numbers discussed in Section \ref{ssec:tangentnumbers}.
\item 
 Theorem ~\ref{twr:CLTCommutatorsAniticommutators} in the special case $a=0$ and $b=1$  leads
to another new  fact about the tangent numbers $T_n$, and the Riemann zeta function
\begin{align*}
 T_{2k+1}&=\lim_{n\to\infty}(2k+1)!\Tr (\P B_n ^{2k} ),\text{ } 
    \zeta(2k+2)=\lim_{n\to\infty}\frac{\pi ^{2k+2}\Tr\big(P_n B_n ^{2k}\big)}{ 2(2^{2k+2}-1)}\quad k\geq 0.\end{align*}
    The approximation of
the values of the Riemann zeta function
for even integers is a popular theme,
see \cite{Williams1971,Apostol1973,Cvijovic2003}. 
It is interesting to evaluate  Theorem~\ref{twr:CLTCommutatorsAniticommutators} for
$a=b=1$. 
Indeed, by the identity $\tan z +\sec z =\frac{1+\tan(z/2)}{1-\tan(z/2)}$, we have 
$$
H_Y=\frac{1}{z}\frac{\tan z }{1-\tan z }-1=\frac{1}{z}\frac{\tan (2z) +\sec (2z)-1}{2}- 1
.
$$
This provides a new aproximation of
the Euler zigzag numbers $E_n$, namely 
\begin{align*}
E_{k}=\lim_{n\to\infty}\frac{k!\Tr\big( \P [\P +B_n]^{k-1}\big)}{2^{k-1}}, \quad k\geq 2.
\end{align*}

\end{enumerate}

\end{Rem}

\section{Measure  and L\'evy-Khinchin representation for the tangent laws}
\noindent
In the last section we compute the measure $\mu$ and the  L\'evy-Khinchin representation 
for the Boolean tangent laws. We focus on the special case  $a=0$ and $b=1$ because in this situation we are able to determine the corresponding measures.    In this case the corresponding transforms are given by  

$$M_\mu(z)=\frac{z }{2z-\tan(z)},\quad G_\mu(z)=\frac{1 }{2z-z^2\tan(1/z)} \text{ and }\phi_{\mu}(z)= z^2\tan(1/z)-z.$$ 

\subsection{The  measure of the tangent law}
The Cauchy
transform determines uniquely the measure and there is an inversion formula called Stieltjes
inversion formula, namely 
\begin{equation*}
d\mu(x)=-\frac{1}{\pi}\lim_{\epsilon \to 0^+}\Im G_{\mu}(x+i\epsilon)=-\frac{1}{\pi}\lim_{\epsilon \to
  0^+} \Im \frac{1 }{2(x+i\epsilon)-(x+i\epsilon)^2\tan(1/(x+i\epsilon))} =0
\end{equation*}
for $2/x\neq \tan(1/x)$ and $x\neq 0$. Thus  the  measure $\mu$ has no absolutely continuous part.
In order to determine the atoms, we compute the  limits 
\begin{equation*}
\lim_{\epsilon \to 0^+}i\epsilon G_\mu\left({x+i\epsilon}\right)=0
\end{equation*}
whenever $x= 0$ or $x$ satisfies
\begin{equation} \label{rownietangent}
2/x= \tan(1/x).
\end{equation}
For $x$ as in \eqref{rownietangent}, we get via de L'Hospital's rule
\begin{align*}
\lim_
{\epsilon \to 0^+} &\frac{i\epsilon}{x+i\epsilon} \frac{1}{2-(x+i\epsilon)\tan(1/(x+i\epsilon))}
\\&=\frac{1}{x}\lim_
{\epsilon \to 0^+} \frac{i}{-i\tan(1/(x+i\epsilon))+(x+i\epsilon)\frac{1}{\cot^2(1/(x+i\epsilon))}\frac{i}{(x+i\epsilon)^2}}
\\&=\frac{1}{x}\frac{1}{-\frac{2}{x}+\frac{1}{x\cos^2(1/x)}}=\frac{\cos^2(1/x)}{-2\cos^2(1/x)+1}=\frac{1}{\tan^2(1/x)-1}
\\&=\frac{1}{(2/x)^2-1}=\frac{x^2}{4-x^2}.
\end{align*}
For $x=0$ we have 
\begin{align*}
\lim_
{\epsilon \to 0^+} &\frac{i\epsilon}{i\epsilon} \frac{1}{2-i\epsilon\tan(1/(i\epsilon))}=\lim_
{\epsilon \to 0^+}  \frac{1}{2+\epsilon\tanh(1/\epsilon)}=\frac{1}{2}.
\end{align*}
\noindent
Note that, in general, the singular part  (= discrete part + singular continuous part) of a probability measure $\mu$  is supported on the
set 
$$S=\{x \in \R \mid \lim_{\epsilon \to 0^+}\Im G_{\mu}(x+i\epsilon)=-\infty\}.$$
See \cite[Page 71]{HasebeFranzSchleissinger} or \cite{AronszajnDonoghue1} for more information. After some calculations, it 
can be shown that in the considered case $S= \{x\mid 2/x=\tan(1/x)\}\cup \{0\}$ or it can be concluded from the previous calculation that on this set we have a  non-zero mass. 
Finally, we infer that the  measure $\mu$ is given by
$$
\mu(\{x\}) =
\begin{cases}
\frac{x^2}{4-x^2}& \text{ for }x\in \{x\mid 2/x=\tan(1/x)\},\\
\frac{1}{2} & \text{ for } x=0.
\end{cases}
$$

\begin{Rem}The above measure is positive. 
By elementary analysis of the function  $f(x)=\tan x - 2x$ on interval $[0,\frac{\pi}{2})$ we see that   $f$ is decreasing for $x\in [0,\frac{\pi}{4})$ and increasing for $x\in (\frac{\pi}{4},\frac{\pi}{2})$. We can also observe that $f(0)=0$,
  $f(\frac{\pi}{4})<0$ and  $\lim_{x\to \frac{\pi}{2}^-}f(x)=\infty$. From this we conclude that $\tan x - 2x=0$ has two roots on $[0,\frac{\pi}{2})$,  first is $0$ and the second one is bigger than $\frac{\pi}{4}$. From this we conclude that the  positive roots of  equation \eqref{rownietangent} satisfy $x <\frac{4}{\pi}$. An analysis of the negative part is  analogous and we infer that solution of  equation \eqref{rownietangent} satisfy $\abs x <\frac{4}{\pi}<2. $
\end{Rem}
\noindent
Figure \ref{fig:tan} presents the function $x\text{tan}(1/x)$. We also plot a constant function equal to 2 and mark the roots of the equation $2/x=\text{tan}(1/x).$ More precisely these roots are approximately as follows: $\pm 0.857956, \pm 0.217192, \pm 0.128372, \pm 0.09132$. After inserting values of these roots into equation $\frac{x^2}{4-x^2}$ and summing the results we get the value approximately equal to $1/2$. Thus for the entire space the measure $\mu$ returns 1. From here we can conclude that above measure $\mu(\{x\})$ is a probability measure.

\begin{figure}[h]
    \centering
    \includegraphics[scale=0.6]{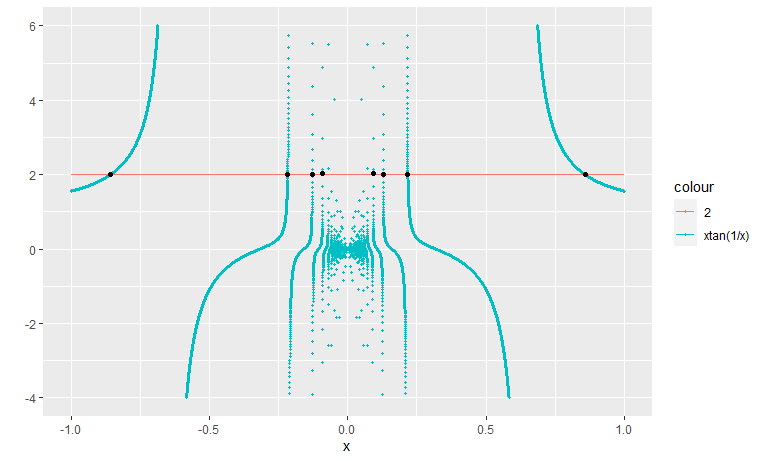}
    \caption{Function $x\text{tan}(1/x)$, constant function equal to 2 and roots of the equation $2/x=\text{tan}(1/x).$}
    \label{fig:tan}
\end{figure}

\subsection{The L\'evy measure of the tangent law}
Now, we will show that   L\'evy measure is given by
$$
\rho(\{x\}) =
\begin{cases}
\frac{x^4}{1+x^2}& \text{ for }x=\frac{2}{n\pi} \text{ with  }n\in\IZ \text{ odd},\\
0 & \text{otherwise.}
\end{cases}
$$
This result will be verified by direct integration  $$\phi_{\mu}(z)= \gamma+\int_{\R}\frac{1+xz}{z-x}\,\dd\rho({x})=z^2\tan(1/z)-z$$
with $\gamma=0.$
We use the well known Euler's partial fraction
expansion of the cotangent function \cite[Ch.~25]{AignerZiegler:2014:BOOK5ed}
\begin{align*}
\cot z &= \frac{1}{z} + \sum_{k=1}^\infty \frac{2z}{z^2-k^2\pi^2}.
\end{align*}
This immediately yields a similar expansion for the tangent function
\begin{align}\label{eq:euler}
\tan \frac{1}{z} = \cot \frac{1}{z}  - 2\cot 2\frac{1}{z} = \sum_{n\in \IN\text{ odd}}\frac{8z}{n^2\pi^2z^2-4}
\end{align}
for $z\neq 0 $ and $z\neq \frac{2}{n\pi}$, $n\in \IN\text{ odd}$.
A direct integration immediately gives
\begin{align*}
\phi_{\mu}(z)&=\int_{\R}\frac{1+xz}{z-x}\,\dd\rho({x})=
  \int_{\R}\left(\frac{1}{z-x}+\frac{x}{1+x^2}\right)(1+x^2)\,\dd\rho({x})
\\&= \sum_{n\in \IZ\text{ odd}} \frac{1}{z-\frac{2}{n\pi}} \frac{16}{n^4\pi^4}+\underbrace{\int_{\R}x\,\dd\rho({x})}_{\text{=0, since~$\rho$ is symmetric}}
\\&= \sum_{n\in \IN\text{ odd}} \frac{32z}{n^2\pi^2z^2-4} \frac{1}{n^2\pi^2}
\\&= \sum_{n\in \IN\text{ odd}}  \left(\frac{8z^3}{n^2\pi^2z^2-4} - \frac{8z}{n^2\pi^2}\right)
\\&=z^2\tan(1/z)-z,
\end{align*}
where in the last line we used the formula \eqref{eq:euler}  and 
$\sum_{n\in \IN\text{ odd}} \frac{1}{n^2}=\frac{\pi^2}{8}$.

\bibliographystyle{amsplain}

\bibliography{sumcomm}

\providecommand{\bysame}{\leavevmode\hbox to3em{\hrulefill}\thinspace}
\providecommand{\MR}{\relax\ifhmode\unskip\space\fi MR }
\providecommand{\MRhref}[2]{%
  \href{http://www.ams.org/mathscinet-getitem?mr=#1}{#2}
}
\providecommand{\href}[2]{#2}
\begin{thebibliography}{10}

\bibitem{AignerZiegler:2014:BOOK5ed}
Martin Aigner and G\"unter~M. Ziegler, \emph{Proofs from {T}he {B}ook}, fifth
  ed., Springer-Verlag, Berlin, 2014.

\bibitem{Andre1880}
D\'{e}sir\'{e} Andr\'{e}, \emph{D\'{e}veloppements, par rapport au module, des
  fonctions elliptiques {$\lambda (x),\;\mu (x)$} et de leurs puissances}, Ann.
  Sci. \'{E}cole Norm. Sup. (2) \textbf{9} (1880), 107--118.

\bibitem{Anshelevich2009}
Michael Anshelevich, \emph{Appell polynomials and their relatives. {II}.
  {B}oolean theory}, Indiana Univ. Math. J. \textbf{58} (2009), no.~2,
  929--968. \MR{2514394}

\bibitem{Apostol1973}
Tom~M. Apostol, \emph{Another elementary proof of {E}uler's formula for {$\zeta
  (2n)$}}, Amer. Math. Monthly \textbf{80} (1973), 425--431.

\bibitem{AHLV2015}
Octavio Arizmendi, Takahiro Hasebe, Franz Lehner, and Carlos Vargas,
  \emph{Relations between cumulants in noncommutative probability}, Adv. Math.
  \textbf{282} (2015), 56--92. \MR{3374523}

\bibitem{Arlinskii}
Yuri Arlinskii, Sergey Belyi, and Eduard Tsekanovskii, \emph{Conservative
  realizations of {H}erglotz-{N}evanlinna functions}, Operator Theory: Advances
  and Applications, vol. 217, Birkh\"{a}user/Springer Basel AG, Basel, 2011.

\bibitem{Arnold:1991}
V.~I. Arnol'd, \emph{Bernoulli-{E}uler updown numbers associated with function
  singularities, their combinatorics and arithmetics}, Duke Math. J.
  \textbf{63} (1991), no.~2, 537--555.

\bibitem{Arnold:1992:snakes}
\bysame, \emph{Snake calculus and the combinatorics of the {B}ernoulli, {E}uler
  and {S}pringer numbers of {C}oxeter groups}, Uspekhi Mat. Nauk \textbf{47}
  (1992), no.~1(283), 3--45, 240.

\bibitem{AronszajnDonoghue1}
N.~Aronszajn and W.~F. Donoghue, \emph{A supplement to the paper on exponential
  representations of analytic functions in the upper half-plane with positive
  imaginary part}, J. Analyse Math. \textbf{5} (1956), no.~321, 113--127.

\bibitem{Bapat}
R.~B. Bapat, \emph{Graphs and matrices}, Universitext, Springer, London;
  Hindustan Book Agency, New Delhi, 2010. \MR{2797201}

\bibitem{BelinschiNica2008}
Serban~T. Belinschi and Alexandru Nica, \emph{On a remarkable semigroup of
  homomorphisms with respect to free multiplicative convolution}, Indiana Univ.
  Math. J. \textbf{57} (2008), no.~4, 1679--1713. \MR{2440877}

\bibitem{BercoviciPata:1999}
Hari Bercovici and Vittorino Pata, \emph{Stable laws and domains of attraction
  in free probability theory}, Ann. of Math. (2) \textbf{149} (1999), no.~3,
  1023--1060, With an appendix by Philippe Biane.

\bibitem{Bozejko1986}
Marek Bo\.{z}ejko, \emph{Positive definite functions on the free group and the
  noncommutative {R}iesz product}, Boll. Un. Mat. Ital. A (6) \textbf{5}
  (1986), no.~1, 13--21. \MR{833375}

\bibitem{BozejkoLeinertSpeicher}
Marek Bo\.{z}ejko, Michael Leinert, and Roland Speicher, \emph{Convolution and
  limit theorems for conditionally free random variables}, Pacific J. Math.
  \textbf{175} (1996), no.~2, 357--388. \MR{1432836}

\bibitem{Carlitz:1975:permutations}
L.~Carlitz, \emph{Permutations, sequences and special functions}, SIAM Rev.
  \textbf{17} (1975), 298--322.

\bibitem{CarlitzScoville:1972}
L.~Carlitz and Richard Scoville, \emph{Tangent numbers and operators}, Duke
  Math. J. \textbf{39} (1972), 413--429.

\bibitem{Cvijovic2003}
Djurdje Cvijovi\'{c}, Jacek Klinowski, and H.~M. Srivastava, \emph{Some
  polynomials associated with {W}illiams' limit formula for {$\zeta(2n)$}},
  Math. Proc. Cambridge Philos. Soc. \textbf{135} (2003), no.~2, 199--209.

\bibitem{Donoghue:1974:monotone}
William~F. Donoghue, Jr., \emph{Monotone matrix functions and analytic
  continuation}, Springer-Verlag, New York-Heidelberg, 1974, Die Grundlehren
  der mathematischen Wissenschaften, Band 207.

\bibitem{Ejsmont2016}
Wiktor Ejsmont, \emph{A characterization of the normal distribution by the
  independence of a pair of random vectors}, Statist. Probab. Lett.
  \textbf{114} (2016), 1--5. \MR{3491965}

\bibitem{EjsmontLehner:2017}
Wiktor Ejsmont and Franz Lehner, \emph{Sample variance in free probability}, J.
  Funct. Anal. \textbf{273} (2017), no.~7, 2488--2520.

\bibitem{EjsmontLehner:2020:tangent}
\bysame, \emph{The free tangent law}, Adv. in Appl. Math. \textbf{121} (2020),
  102093, 32. \MR{4136165}

\bibitem{EjsmontLehner:2019:commutators}
\bysame, \emph{Sums of commutators in free probability}, J. Funct. Anal.
  \textbf{280} (2021), no.~2, 108791, 29. \MR{4159271}

\bibitem{EjsmontLehner:2019:cotangent}
\bysame, \emph{The trace method for cotangent sums}, J. Combin. Theory Ser. A
  \textbf{177} (2021), 105324, 32. \MR{4143530}

\bibitem{Evans}
Steven~N. Evans, \emph{Probability and real trees}, Lecture Notes in
  Mathematics, vol. 1920, Springer, Berlin, 2008, Lectures from the 35th Summer
  School on Probability Theory held in Saint-Flour, July 6--23, 2005.
  \MR{2351587}

\bibitem{FevrierMastnakNicaSzpojankowski2020}
Maxime Fevrier, Mitja Mastnak, Alexandru Nica, and Kamil Szpojankowski,
  \emph{Using {B}oolean cumulants to study multiplication and anti-commutators
  of free random variables}, Trans. Amer. Math. Soc. \textbf{373} (2020),
  no.~10, 7167--7205. \MR{4155204}

\bibitem{Flajolet1980}
P.~Flajolet, \emph{Combinatorial aspects of continued fractions}, Discrete
  Math. \textbf{32} (1980), no.~2, 125--161.

\bibitem{Franz2009}
Uwe Franz, \emph{Monotone and {B}oolean convolutions for non-compactly
  supported probability measures}, Indiana Univ. Math. J. \textbf{58} (2009),
  no.~3, 1151--1185. \MR{2541362}

\bibitem{HasebeFranzSchleissinger}
Uwe Franz, Takahiro Hasebe, and Sebastian Schlei\ss~inger, \emph{Monotone
  increment processes, classical {M}arkov processes, and {L}oewner chains},
  Dissertationes Math. \textbf{552} (2020), 119. \MR{4152669}

\bibitem{Gesztesy}
Fritz Gesztesy and Eduard Tsekanovskii, \emph{On matrix-valued {H}erglotz
  functions}, Math. Nachr. \textbf{218} (2000), 61--138.

\bibitem{GrahamKnuthPatashnik:1994}
Ronald~L. Graham, Donald~E. Knuth, and Oren Patashnik, \emph{Concrete
  mathematics}, second ed., Addison-Wesley Publishing Company, Reading, MA,
  1994.

\bibitem{Helms}
Lester~L. Helms, \emph{Introduction to probability theory with contemporary
  applications}, second ed., Universitext, Courier Corporation, 2012.

\bibitem{HiwatashiKurodaNagisaYoshida:1999a}
Osamu Hiwatashi, Tomoko Kuroda, Masaru Nagisa, and Hiroaki Yoshida, \emph{The
  free analogue of noncentral chi-square distributions and symmetric quadratic
  forms in free random variables}, Math. Z. \textbf{230} (1999), no.~1, 63--77.

\bibitem{Hoffman:1999:derivative}
Michael~E. Hoffman, \emph{Derivative polynomials, {E}uler polynomials, and
  associated integer sequences}, Electron. J. Combin. \textbf{6} (1999),
  Research Paper 21, 13.

\bibitem{James:1958}
G.~S. James, \emph{On moments and cumulants of systems of statistics},
  Sankhy\=a \textbf{20} (1958), 1--30.

\bibitem{KaganShalaevski}
A.~M. Kagan and O.~V. \v{S}alaevski\u{\i}, \emph{A characterization of the
  normal law by a property of the non-central chi-square distribution},
  Litovsk. Mat. Sb. \textbf{7} (1967), 57--58. \MR{0220329}

\bibitem{Kreweras:1972}
G.~Kreweras, \emph{Sur les partitions non crois\'ees d'un cycle}, Discrete
  Math. \textbf{1} (1972), no.~4, 333--350.

\bibitem{Lehner2002}
Franz Lehner, \emph{Free cumulants and enumeration of connected partitions},
  European J. Combin. \textbf{23} (2002), no.~8, 1025--1031. \MR{1938356}

\bibitem{Lehner:2003}
\bysame, \emph{Cumulants in noncommutative probability theory. {II}.
  {G}eneralized {G}aussian random variables}, Probab. Theory Related Fields
  \textbf{127} (2003), no.~3, 407--422.

\bibitem{Lehner:2004}
\bysame, \emph{Cumulants in noncommutative probability theory. {I}.
  {N}oncommutative exchangeability systems}, Math. Z. \textbf{248} (2004),
  no.~1, 67--100.

\bibitem{LeonovShiryaev:1959}
V.~P. Leonov and A.~N. Shiryaev, \emph{On a method of calculation of
  semi-invariants}, Theor. Prob. Appl. \textbf{4} (1959), 319--328.

\bibitem{Liu2015}
Weihua Liu, \emph{A noncommutative de {F}inetti theorem for boolean
  independence}, J. Funct. Anal. \textbf{269} (2015), no.~7, 1950--1994.
  \MR{3378866}

\bibitem{Liu2015b}
\bysame, \emph{Extended de {F}inetti theorems for boolean independence and
  monotone independence}, Trans. Amer. Math. Soc. \textbf{370} (2018), no.~3,
  1959--2003. \MR{3739198}

\bibitem{NicaSpeicher:1998}
Alexandru Nica and Roland Speicher, \emph{Commutators of free random
  variables}, Duke Math. J. \textbf{92} (1998), no.~3, 553--592.

\bibitem{NicaSpeicher:2006}
\bysame, \emph{Lectures on the combinatorics of free probability}, London
  Mathematical Society Lecture Note Series, vol. 335, Cambridge University
  Press, Cambridge, 2006.

\bibitem{PitmanYor2003}
Jim Pitman and Marc Yor, \emph{Infinitely divisible laws associated with
  hyperbolic functions}, Canad. J. Math. \textbf{55} (2003), no.~2, 292--330.
  \MR{1969794}

\bibitem{Popa2009}
Mihai Popa, \emph{A new proof for the multiplicative property of the {B}oolean
  cumulants with applications to the operator-valued case}, Colloq. Math.
  \textbf{117} (2009), no.~1, 81--93. \MR{2539549}

\bibitem{PopaHao2019}
Mihai Popa and Zhiwei Hao, \emph{An asymptotic property of large matrices with
  identically distributed {B}oolean independent entries}, Infin. Dimens. Anal.
  Quantum Probab. Relat. Top. \textbf{22} (2019), no.~4, 1950024, 22.
  \MR{4065491}

\bibitem{Speicher:1998}
Roland Speicher, \emph{Combinatorial theory of the free product with
  amalgamation and operator-valued free probability theory}, Mem. Amer. Math.
  Soc. \textbf{132} (1998), no.~627, x+88.

\bibitem{SpeicherWorudi1995}
Roland Speicher and Reza Woroudi, \emph{Boolean convolution}, Free probability
  theory ({W}aterloo, {ON}, 1995), Fields Inst. Commun., vol.~12, Amer. Math.
  Soc., Providence, RI, 1997, pp.~267--279. \MR{1426845}

\bibitem{StanleyVol2}
Richard~P. Stanley, \emph{Enumerative combinatorics. {V}ol. 2}, Cambridge
  Studies in Advanced Mathematics, vol.~62, Cambridge University Press,
  Cambridge, 1999.

\bibitem{Stanley2010}
\bysame, \emph{A survey of alternating permutations}, Combinatorics and graphs,
  Contemp. Math., vol. 531, Amer. Math. Soc., Providence, RI, 2010,
  pp.~165--196.

\bibitem{Steinmetz}
Norbert Steinmetz, \emph{Nevanlinna theory, normal families, and algebraic
  differential equations}, Universitext, Springer, Cham, 2017.

\bibitem{VoiculescuDykemaNica:1992}
D.~V. Voiculescu, K.~J. Dykema, and A.~Nica, \emph{Free random variables}, CRM
  Monograph Series, vol.~1, American Mathematical Society, Providence, RI,
  1992.

\bibitem{Voiculescu:1986}
Dan Voiculescu, \emph{Addition of certain noncommuting random variables}, J.
  Funct. Anal. \textbf{66} (1986), no.~3, 323--346.

\bibitem{Voiculescu:1991}
\bysame, \emph{Limit laws for random matrices and free products}, Invent. Math.
  \textbf{104} (1991), no.~1, 201--220.

\bibitem{Williams1971}
Kenneth~S. Williams, \emph{On {$\sum _{n=1}^{\infty }\ (1/n^{2k})$}}, Math.
  Mag. \textbf{44} (1971), 273--276.

\end{thebibliography}
\end{document}